\documentclass[11pt,reqno,letter]{amsart}

\usepackage{amsmath,amsthm,verbatim,amscd,amssymb,setspace,enumitem,hyperref,exscale,color,slashed}
\usepackage[all,pdf]{xy}
\allowdisplaybreaks[4]

\newcommand{\CC}{\mathbb{C}}
\newcommand{\RR}{\mathbb{R}}

\newcommand{\K}{\mathcal{K}}
\newcommand{\J}{\mathcal{J}}
\newcommand{\im}{\operatorname{Im}}
\newcommand{\C}{\mathcal{C}}

\newcommand{\m}{\mathfrak{m}}

\newcommand{\Ric}{\textup{Ric}}

\newcommand{\vol}{\operatorname{Vol}}

\newcommand{\re}{\operatorname{Re}}
\newcommand{\tr}{\operatorname{tr}}

\newtheorem{thm}{Theorem}[section]
\newtheorem{pro}[thm]{Proposition}
\newtheorem{cor}[thm]{Corollary}
\newtheorem{lem}[thm]{Lemma}

\newcounter{mtheorem}
\newtheorem{mtheorem}[mtheorem]{Theorem}

\theoremstyle{definition}

\newtheorem{rem}[thm]{Remark}
\newtheorem{ex}[thm]{Example}

\numberwithin{equation}{section}

\title{Mass and Expansion of Asymptotically Conical K\"ahler Metrics}

\author{Qi Yao}
\date{}

\begin{document}

\maketitle
\thispagestyle{empty}
\begin{abstract} 
    We prove an expansion theorem on scalar-flat asymptotically conical (AC) K\"ahler metrics. The ADM mass is known to be defined in ALE Riemannian manifolds. The concept of ADM mass can be generalized to AC cases. Consider an AC K\"ahler manifold with asymptotic to a Ricci-flat K\"ahler metric cone with complex dimension $n$. Assuming the weak decay conditions required for the mass to be well-defined, then each scalar-flat AC K\"ahler metric admits an expansion that the main term is given by the standard K\"ahler metric of the metric cone and the leading error term is of $O(r^{2-2n})$ with coefficient only depending on the ADM mass and its dimension. Besides, the mass formula by Hein-LeBrun  \cite{hein2016mass} also can be proved in our setting. As an interesting application, a new version of the positive mass theorem will also be discussed in the cases of the resolutions of the Ricci-flat K\"ahler cones.

\end{abstract}

\section{Introduction and Statement of Main Theorems} Schwarzschild metric is known to be a solution of Einstein's vacuum equations. In general space-time $\RR^{n+1}$ $(n\geq 3)$, the Schwarzschild metric has the following explicit expression.
\begin{align*}
    g = - (1- \frac{2m}{r^{n-2}})dt^2 + (1-\frac{2m}{r^{n-2}}) dr^2 + r^2 g_{S^{n-1}},
\end{align*}
where the constant $m$ is the mass of the the Schwarzschild metric. One can easily observe that, by restricting to each space slice $t=c$, the Schwarzschild metric differs from the Euclidean metric by quantity of order $O(r^{2-n})$. 

Motivated by Schwarschild metrics, Arnowitt, Deser and Misner \cite{PhysRev.122.997} first generalized the definition of  the mass (ADM mass) to asymptotically flat metrics. Here, we apply the definition of mass to asymptotically locally Euclidean (ALE) manifolds. Let $(X,g)$ be an complete non-compact Riemannian manifold, we say $(X,g)$ is ALE if there is a compact subset $K \subseteq X$ such that $X-K$ has finitely many components, denoted by $X_{i,\infty}$ with $X-K = \bigcup_{i} X_{i,\infty}$ and each $X_{i,\infty}$ is diffeomorphic to $ (\RR^n-B_R)/\Gamma_i$, where $B_R$ is a closed ball of radius $R$ and $\Gamma$ is a finite subset of $O(n)$. The metric $g$ is asymptotic to the Euclidean metric with rapid decay rate. Then, the ADM mass of the end $X_{i,\infty}$ is given by
\begin{align} \label{admmass1}
   m_i(g)=  \lim_{r \rightarrow \infty} \frac{\Gamma(\frac{n}{2})}{4(n-1)\pi^{n/2}}  \int_{S_r/\Gamma_i} (g_{ij,i}-g_{ii,j})n^j d\mu
\end{align}
where $S_r$ is the Euclidean sphere of radius $r$, $n$ is the outward Euclidean unit normal vector and $d\mu$ is the volume form induced by standard metric on $S_r$. Bartnik \cite{bartnik1986mass} and Chru\'{s}ciel \cite{Chruciel1985BoundaryCA} give the appropriate decay condition to make the ADM mass   coordinate-invariant, for Sobolev case and H\"older case respectively. Here, we list the decay condition of H\"older case for the metric $g$ on each end $X_{i,\infty}$,
\begin{enumerate}
\item[(i)] the scalar curvature $R$ of $g$ belongs to $L^1$.
\item[(ii)] the metric $g$ is asymptotic to the Euclidean metric $\delta_{ij}$ at the end with decay rate $-\tau$ for some $\tau>(n-2)/2$, 
\begin{align} \label{decayale}
g_{ij}= \delta_{ij} + O(r^{-\tau}) , \qquad  |\nabla  (\psi^{-1})^* g) |_{g_0} = O(r^{-\tau-1}).
\end{align}
\end{enumerate}

In this paper, one of main results is that the scalar-flat ALE K\"ahler metric of complex dimension $n$ is of decay rate $2-2n$, which has an expansion, $g_{ij} = \delta_{ij} + O(r^{2-2n})$ and the leading error term is determined by the ADM mass of $g$ (see theorem \ref{expanthm}). Furthermore, the expansion also works in scalar-flat AC K\"ahler cases if we replace the Euclidean metric by the standard metric on K\"ahler cone. 

Rather than directly introduce the expansion theorem, we first introduce an important tool in K\"ahler geometry, the ddbar lemma. The ddbar lemma is a standard result in compact K\"ahler manifolds and it can be easily proved in ALE K\"ahler manifolds if we assume a fast decay condition (see \cite[theorem 8.4.4]{joyce2000compact}). In Colon-Hein \cite[theorem 3.11]{conlon2013asymptotically}, the ddbar lemma is generalized to asymptotically conical (AC) K\"ahler manifolds with a lower decay condition $(\text{only need }-\tau <0)$, with additional assumption of non-negative Ricci curvature. In the author's previous paper \cite[Proposition 3.6]{yao2020invariant}, the ddbar lemma can be proved in negative line bundle over K\"ahler C-spaces (compact simply-connected homogeneous K\"ahler manifolds) without any decay assumption at infinity. Here, we give a new version of ddbar lemma by revising the Colon-Hein's result and dropping the non-negative Ricci curvature assumption. 

The ddbar lemma on the complement of $K\subseteq X$ will also be discussed in this paper. In Goto \cite[section5]{goto2012calabi} and Conlon-Hein \cite[appendix A]{conlon2013asymptotically}, ddbar lemma is proved in the case of complex dimension $n \geq 3$ with a trivial canonical bundle. In this paper, we derive a ddbar lemma on $X\backslash K$ without making the above assumptions, but with high decay rate error terms.

Let $(L, g_L)$ be a compact Riemannian manifold, the \textit{Riemannian cone} $C_L$ associated with $L$ is  defined to be $L\times \RR_{>0}$ with Riemannian metric $g_{0} = dr^2 + r^2 g_L$. A Riemannian cone $C_L$ is said to be \textit{K\"ahler} if there exists a $g_{0}$-parallel complex structure $J_0$ such that the corresponding fundamental form $\omega_0 = g_0(J_0\cdot, \cdot)$ is closed, in particular, $\omega_0 = i\partial \overline{\partial} r^2$. If we assume that the K\"ahler cone $(C_L, g_0, J_0)$ is Ricci-flat with $\dim_\CC C_L =n$,
then, according to the standard calculation in \cite[section 1.4]{sparks2010sasaki} and \cite[section 11.1]{boyer2008sasakian},  the link $L$ is Sasaki-Einstein with $\Ric g_L = 2(n-1) g_L$. 

Let $(X, J, g)$ be an AC K\"ahler manifold asymptotic to $C_L$, where $(C_L, g_0)$ is a Ricci-flat K\"ahler cone with link $L$. Throughout the paper we always assume that the manifolds only have one end, then there exists a compact subset $K\subseteq X$ and $B_R =\{x\in C_L, r(x) < R\}$, such that  $\psi: X-K \rightarrow C_L-B_{R}$ is a diffeomorphism satisfying the following decay conditions,
\begin{enumerate}
\item[(i)] the scalar curvature $R$ of $g$ belongs to $L^1$.
\item[(ii)] the complex structure $J$ on $X$ decays to $J_0$, the induced almost complex structure from $C_L$.  
\item[(iii)] the metric $g$ is asymptotic to the reference metric $g_L$ at the end with decay rate $-\tau $ for some $\tau>0$, for $i=0,1,\ldots,k$,
\begin{align} \label{decayac}
| \nabla^i  ((\psi^{-1})^* g - g_0)|_{g_0} = O(r^{-\tau-i}).
\end{align}
\end{enumerate}

\begin{mtheorem} (ddbar lemma) \label{ddclem}
Let $X$ be AC K\"ahler manifolds asymptotic to a Ricci-flat K\"ahler cone $C_L$. Let $k$ be a large positive integer, $\alpha \in (0,1)$ and $\delta > 0$,
\begin{enumerate}
\item[(i)] Let $\omega$ be a $d$-exact real $(1,1)$-form on $X$ satisfying the decay condition $\omega \in \C^{k,\alpha}_{-\delta}$. Then, there exists a real function $\varphi \in \C^{k+2,\alpha}_{2-\delta}$ such that $\omega = dd^c \varphi$.
\item[(ii)] Let $\omega$ be a $d$-exact real $(1,1)$-form on the end $X\backslash K$.  Then there exists a real function $\varphi \in \C^{k+2, \alpha}_{2-\delta}(X\backslash K)$ such that,
\begin{align*}
    \omega = dd^c \varphi + O(r^{-2n}).
\end{align*}
\end{enumerate}
\end{mtheorem}

We generalize the definition of mass to AC Riemannian manifolds. Let $(C_L, g_0)$ be a Riemannian metric cone of link $(L, g_L)$ and $(X,g)$, a complete Riemannian manifold asymptotic to $(C_L, g)$ at infinity. The definition in (\ref{admmass1}) requires a coordinate system at infinity, which does not exist in general AC Riemannian manifolds. Hence, we start from a coordinate-free expression of mass. In Lee's book \cite[Section 3.1.3]{lee2019geometric}, the mass is defined to be
\begin{align} \label{admmass2}
    \m(g) =  \frac{1}{2(2n-1) \vol(L)}\lim_{r \rightarrow \infty} \int_{L(r)} (\overline{\nabla}^j g_{ij}- (\tr_{g_0}g)_i) n^i d\vol_{L(r)},
\end{align}
where $\overline{\nabla}$ is the Levi-Civita connection with respect to $g_0$ and $n$ is the outer normal vector field on $L(r)$. In Hein-LeBrun \cite{hein2016mass}, the ADM mass in ALE K\"ahler manifolds is a quantity determined by the total scalar curvature and its topological data. A similar mass formula also holds in AC K\"ahler manifolds asymptotic to Ricci-flat K\"ahler cones.  Let $\iota: H^2_c (X) \rightarrow H^2_{dR}(X)$ be the map induced by the natural embedding of chain complex $\Omega^{\bullet}_{c} \hookrightarrow \Omega^{\bullet}_{dR}$. Generally speaking, $\iota$ is not an isomorphism, but the  first Chern class of $(X,J)$ always has a pre-image under $\iota$ (see Lemma \ref{cptclasslm}). Then, we have the following mass formula.

\begin{mtheorem}
Let $(X, J, g)$ be an AC K\"ahler manifolds asymptotic to a Ricci-flat K\"ahler cone $(C_L, J_0, g_0)$ satisfying the decay condition (\ref{decayac}) with $\tau = n-1+ \epsilon$. Then, we have
\begin{align} \label{admmass3}
\m( g) = -\frac{2\pi \langle \iota^{-1} c_1, [\omega]^{n-1} \rangle}{(2n-1) (n-1)! \vol(L)} + \frac{1}{2(2n-1) \vol(L)} \int_{X} R_g d\vol_g,
\end{align}
where $[\omega]$ is the K\"ahler class of $(g,J)$, $c_1$ is the first Chern class of $(X,J)$ and $R_g$ is the scalar curvature of $g$ on $X$.
\end{mtheorem}

An interesting application of the mass formula (\ref{admmass3}) is the positive mass theorem. The positive mass theorem is first proved in AE manifolds of lower dimension by Schoen, Yau \cite{schoen1979complete}. Afterwards, Witten \cite{witten1981new} develop a new method to prove the positive mass theorem on spin manifolds. And in K\"ahler cases, Hein-Lebrun \cite{hein2016mass} confirms the positive mass theorem on AE K\"ahler manifolds. However, in ALE cases, Lebrun \cite{Lebrun1988CounterexamplesTT} constructs the first counter-example of positive mass theorem. Here, we discuss a new version of positive mass theorem on resolution spaces of Ricci-flat K\"ahler cones with an isolated canonical singularity at vertex. 

\begin{mtheorem} (positive mass theorem)
  Let $(X, J)$ be a resolution space of a Ricci-flat K\"ahler cone $C_L$ such that the only singularity $O \in C_L$ is canonical. If $(X,g)$ has scalar curvature $R \geq 0$, then the mass $\m(X,g) \geq 0$, and equals zero only if $(X,J,g)$ is a crepant resolution of $C_L$ with a scalar-flat K\"ahler metric $g$.
\end{mtheorem}

Now, we can introduce the following expansion theorem in AC K\"ahler case. 

\begin{mtheorem} \label{expanthm} Let $(X, J)$ be a AC K\"ahler manifold asymptotic to a Ricci-flat K\"ahler cone $(C_L, J_0)$ and we assume that $k$ be a large positive integer and $\tau = n-1+\epsilon$,
\begin{enumerate}
\item[(i)] Let $\omega_1$, $\omega_2$ be K\"ahler forms on $X$ with the corresponding metrics satisfying decay condition (\ref{decayac}) and $R_1$, $R_2$, the scalar curvature of $\omega_1$ and $\omega_2$. If 
\begin{enumerate}
    \item[$\bullet$] $[\omega_1] = [\omega_2]$ 
    \item[$\bullet$] $R_1 =R_2$, 
\end{enumerate}
then  $\omega_2 = \omega_1 + dd^c \varphi$ with the potential $\varphi \in \C^{k+2,\alpha}_{2-2\tilde{\tau}}$, for some $\tilde{\tau} >n-1$ depending on $(n, L, \tau)$.
\item[(ii)] Let $\omega$ be a scalar flat K\"ahler form on X satifying decay condition (\ref{decayac}). And the complex structure $J$ is asymptotic to $J_0$ satisfying 
\begin{align} \label{cxdec}
J= J_0+ O(r^{2-2n-\epsilon'}), \qquad \epsilon' >0,
\end{align}
Then, the scalar flat K\"ahler form $\omega$ admits the following expansion at infinity. In particular, outside a compact set of $X$, for complex dimension $n\geq 3$, we have
\begin{align*}
\omega = \frac{1}{2} dd^c r^2 + \frac{(2n-1)\m(\omega)}{2(4-2n)(n-1)} dd^c r^{4-2n} + O(r^{-2\tilde{\tau}}),
\end{align*}
where $m(\omega)$ is the ADM mass of $\omega$ and $\tilde{\tau} > n-1$ only depends on $(n, L, \epsilon', \tau)$. And for complex dimension $n=2$, we have
\begin{align*}
\omega = \frac{1}{2} dd^c r^2 +\frac{3 \m(\omega)}{2} dd^c  \log r + O(r^{-2\tilde{\tau}}),
\end{align*}
where $\tilde{\tau} = \min\{\tau, 3/2\}$.
\end{enumerate}
\end{mtheorem}

As special cases of AC K\"ahler metrics, all ALE K\"ahler metrics satisfy the expansion theorem. According to the statement (i) of Theorem \ref{expanthm} in ALE K\"ahler cases,  we can define K\"ahler potential spaces with relatively "good" decay rate. Precisely, by fixing a K\"ahler metric $\omega$, then, we can define the following potential space of ALE K\"ahler metrics,
\begin{align*}
    \mathcal{H}_{-\tau} ([\omega]) = \Big\{f \in \C^\infty_{-2 \tilde{\tau}}: \omega + i\partial\overline{\partial} f >0, \ \tilde{\tau} = \min\{\tau, n-\frac{1}{2}\} \Big\}
\end{align*}
In fact, $\mathcal{H}_{-\tau} ([\omega])$ does not contain all ALE K\"ahler metrics, but it is enough for "prescribed scalar curvature" problem, as all ALE K\"ahler metrics with the same scalar curvature as $\omega$ are contained in $\mathcal{H}_{-\tau} ([\omega])$. For the statement (ii), in ALE K\"ahler case, the fall-off condition of complex structure (\ref{cxdec}) is automatically satisfied. In fact, for $n\geq 3$, in the asymptotic chart,  the complex structure $J$ coincides with the standard one $J_0$, and for $n=2$,  in the asymptotic chart,
$J = J_0 + O(r^{-3})$. One can check \cite[Lemma 2.3, Proposition 4.5]{hein2016mass} for details.

This paper is a part of Ph.D thesis of the author. The author would like to thank Professor Hans-Joachim Hein and Professor Bianca Santoro for suggesting the problem, and for constant support, many helpful comments, as well as much enlightening conversation. This work is completed while the author is supported by scholarship from University of M\"unster, WWU.

\section{Preliminary on Analysis} \label{pres}
\subsection{The Laplacian of 1-form on Riemanian cones} \label{laequ1fss}

This subsection is mainly dedicated to preparing for the proof of Theorem \ref{ddclem}. We will solve the laplacian equation of $1$-form on Riemannian cones. The fact (lemma \ref{laequ1f}) was also claimed in \cite[Lemma 3.7(ii)]{conlon2013asymptotically}. Here, we just give a detailed proof for this lemma.

Let $L$ be a closed Riemannian manifold with dimension $n-1$ ($ n\geq 4$) and $C(L) = L \times \RR_{>0}$ with standard cone metric $g_{C}= dr^2 + r^2 g$. Throughout this paper, we will apply the weighted H\"{o}lder norms on Riemannian cones. Let $T$ be a tensor field on the Riemannian cone $(C_L, g)$ and $U$ be an open domain in $C_L$, the weighted H\"older norm with order $\rho$ is defined to be the following,
\begin{align} \label{hnorm}
||T||_{C^{k,\alpha}_{\rho} (U)} = \sum_{i=0}^{k} \sup_{x \in U} \big|(r^2 +1)^{\frac{1}{2} (-\rho+i)} \nabla^i T \big|_{g} + \sup_{x, y \in U} (r^2 +1)^{\frac{1}{2}(-\rho+k+\alpha)} \frac{|\nabla^{k}T(x)- \nabla^k T(y)|_g}{|x-y|^\alpha}
\end{align}
where $\nabla$ is the Levi-Civita connection of $g$ and $|T(x)-T(y)|_g$ is defined via parallel transport minimal geodesic from $x$ to $y$. A tensor field $T$ is belong to $\C^{k,\alpha}_\rho(U)$ if the H\"older norm $||T||_{\C^{k,\alpha}_\rho(U)}$ is finite. Similarly, we can define a weighted H\"older norm on each AC Riemannian manifold as in (\ref{hnorm}) by fixing a smooth radial function $\tilde{r}$, which is obtained by smoothly extending the radial function $r$ to the whole manifold.

Let $\Delta = dd^* + d^* d$ be the Hodge-Laplacian operator on $C_L$ and $\Delta_L = d_L d^*_L + d_L^* d_L$,  the Hodge-Laplacian operator on $L$.
Let $0= \lambda'_0 < \lambda'_1 \leq \lambda'_2 \leq \ldots $ (listed with multiplicity) be the increasing sequence of eigenvalues of $\Delta_L$ on functions and $\kappa_0$ (constant), $\kappa_1, \ldots$, the corresponding eigenfuncions satisfying $||\kappa_i||_{L^2(L)} =1$, for all $i\geq 0$. Viewing $\Delta_L$ as an operator acting on the spaces of 1-forms on $L$, we immediately get a family of eigen-1-forms, $\Delta_L d_L \kappa_i = \lambda'_i d_L \kappa_i$, $i \geq 1$.  The $L^2$-normalization of $\kappa_i$ implies that $||d_L\kappa_i||_{L^2(L)}= (\lambda'_i)^{1/2} $. Besides, according to Hodge decomposition on $L$, we also have a family of coclosed eigen-1-forms, $\Delta_L \eta_j = \lambda''_j \eta_j$ $(d^*_L \eta_j =0)$, where the eigenvalues are listed as an increasing sequence, $0\leq \lambda''_1 \leq \lambda''_2 \leq \ldots$ and $||\eta_j||_{L^2(L)} =1$. In summary, we have
\begin{enumerate}
\item[(i)] Exact eigenforms: $ d_L\kappa_i$ with eigenvalue $\lambda'_i$ and $||d_L\kappa_i||_{L^2(L)}= \lambda_i^{1/2} $, for $i \geq 1$. 
\item[(ii)] Coclosed eigenforms: $\eta_j$ with eigenvalue $ \lambda''_j $ and $||\eta_j||_{L^2(L)}=1$.
\end{enumerate}
Consider an 1-form $\beta$ defined on $C(L)$, then we  can write $\beta = \kappa (r, y) dr + \eta (r,y) $. By spectral decomposition of 1-form on $L$, we obtain the Fourier series of $\kappa(r,y)$ and $\eta(r,y)$,
\begin{align} \label{dec1f}
\begin{split}
\kappa (r,y) &= \sum_{i\geq 0} f_i (r) \kappa_i (y)\\
\eta (r,y) &= \sum_{i\geq 1} g_i (r) d_L \kappa_i (y) + \sum_{j \geq 1} h_j (r) \eta_j (y).
\end{split}
\end{align}
Based on calculation in \cite[(2.14-2.15)]{cheeger1994cone} or \cite[(3.8)]{cheeger1983spectral}, we have an explicit formula for the Laplacian of 1-form $ \beta = \kappa dr + \eta$. In particular,
\begin{equation}\label{la1f}
\begin{split}
\Delta \beta = \ &dr \big( -\kappa''  -\frac{n-1}{r} \kappa' + \frac{n-1}{r^2} \kappa + \frac{1}{r^2} \Delta_{L} \kappa + \frac{2}{r^3} d^*_{L}  \eta\big)  \\
 & -\eta'' -\frac{n-3}{r} \eta' -\frac{2}{r} d_L \kappa + \frac{1}{r^2} \Delta_L \eta.
\end{split}
\end{equation}
Regarding the Laplacian equation of 1-form, we introduce exceptional sets $A,\ B,\ C$ defined in (\ref{exset1}), (\ref{exset2}), (\ref{exset3}) respectively; namely, the set of orders of homogeneous harmonic 1-forms of three different types. If we write $U(r_0) =\{x \in C_L, r(x) > r_0\}$ and $\overline{U}(r_0)$ represents its topological closure, then we have the following lemma,

\begin{lem} \label{laequ1f} Let $\theta$ be an 1-form defined on $\overline{U}(1)$ with $\theta \in \C^{k,\alpha}_{\rho}$. If $k \geq 2n+3$ and $\rho+2 \notin A \bigcup B \bigcup C$, then there exists  a solution $\beta \in \C^{k+2,\alpha}_{\rho+2} (\overline{U}(1))$ satisfying the Laplacian equation, $\Delta \beta = \theta$ and $\beta$ satisfies the estimate.
\begin{align*}
||\beta||_{\C^{k+2,\alpha}_{\rho+2} (U (2))}  \leq C ||\theta||_{\C^{k, \alpha}_{\rho}(U (1))}.
\end{align*} 
where $C$ only depends on $(n, L, k, \rho)$.
\end{lem}

\begin{proof} One of key points to solve the Laplacian equation of 1-form is to observe that the Laplacian equation is equivalent to a system of ODEs based on (\ref{dec1f}) and (\ref{la1f}). By spectrum decomposition of the Laplacian, $\beta= \kappa dr + \eta$ can be written as in (\ref{dec1f}) and similarly,  we can also represent $\theta= \theta_0 dr + \theta_1$ as,
\begin{align*}
\theta_0 = \sum_{i\geq 0} u_i (r) \kappa_i (y),  \qquad \theta_1 = \sum_{i \geq 1} v_i (r) d _L \kappa_{i} (y)  + \sum_{j\geq 1} w_j  (r) \eta_j (y),
\end{align*}
where $v_i,\  w_j \in \C^{{k,\alpha}}_{\rho +1}$ and $u_i \in \C^{k,\alpha}_{\rho}$. According to (\ref{la1f}), by comparing the coefficients, the Laplacian equation $\Delta \beta = \theta$ is equivalent to the following system of ODEs,
\begin{align}
-h''_{j} - \frac{n-3}{r} h'_j + \frac{\lambda''_j}{ r^{2}} h_j = w_j, \quad (j \geq 1)  \label{laequ1}
\end{align}
and
\begin{subequations}
\begin{align}
 &-f''_i - \frac{n-1} {r} f'_i + \frac{ n-1 +\lambda'_i }{ r^2} f_i - \frac{2\lambda_i'}{ r^3} g_i= u_i,  \label{laequ2} \\
 &-g''_i - \frac{ n-3 }{ r } g'_i + \frac{ \lambda'_i }{r^2} g_i -\frac{2}{ r } f_i = v_i, \quad (i \geq 0)  \label{laequ3}
\end{align}
\end{subequations}   
To solve equation (\ref{laequ1}), notice that the corresponding homogeneous equation has the solutions $r^{a_j^{\pm}}$, where the orders are given by
\begin{align} \label{exset1}
A= \{a_j^{\pm}-1, j\geq 1 \}, \qquad a^{\pm}_{j} = -\frac{n-4}{2} \pm \sqrt{ \Big( \frac{n-4}{2}\Big)^2  +\lambda''_j},
\end{align}
where $A$ defines the first exceptional set. If we write $\displaystyle H_j= \begin{pmatrix}h_{j} \\ h'_{j} \end{pmatrix} $, then we have the following representation formula,
\begin{align} \label{sol1}
H_j (r) = W_{j}(r) \Big( W_j^{-1}(1) H_j(1) +  \int_{1}^r W_{j}^{-1}(s) \begin{pmatrix} 0\\ w_{j} (s) \end{pmatrix} ds  \Big)
\end{align}
where the Wronskian $W_\lambda$ is given by
\begin{align*}
W_{j} (r) = \begin{pmatrix}  r^{a^{+}_{j}} & r^{a^-_{j}} \\ a^+_{j} r^{a^{+}_{j}-1} & a^-_{j} r^{a^{-}_{j}-1} \end{pmatrix}.
\end{align*}
Noting that te exceptional set can be ordered nondecreasingly as $\ldots \leq  a_2^- \leq a_1^- \leq a_1^+ \leq a_2^+ \leq \ldots$.  Without loss of generality, assume that $\rho  > 1-n$ (to ensure $\rho +3 > a_j^-$ for $j\geq 1$). Then, we can deduce the following explicit formula of $h_{j}$,
\begin{align}\label{sol2}
h_j(r) = \hat{h}_j (r) + A_{+} r^{a^+_{j}}+ A_{-} r^{a^-_{j}}.
\end{align}
In case of $ a_j^+ > \rho+3 $, $\hat{h}_\lambda$, $A_+$ and $A_-$ are given as follows, 
\begin{align}
\hat{h}_j &= \frac{1}{a^-_j - a^+_j} \bigg( r^{a^+_j} \int_{\infty}^r  s^{1-a^+_j} w_j(s) ds - r^{a^-_j} \int_{r_0}^r  s^{1-a^-_j} w_j (s) ds\bigg), \nonumber \\
A_{-} &=\frac{1}{a^{-}_{j} - a^{+}_{j} } \big( - a^+_{j} h_{j}(1) +  h'_{j}(1) \big), \label{coefsol1} \\
A_{+} &= \frac{1}{a^-_{j}-a^+_{j}}\bigg( a^-_{j}  h_{j}(1) - h'_{j}(1) + \int_{1}^\infty s^{1-a^+_{j}} w_j (s) ds \bigg). \nonumber
\end{align}
It is obvious to see that $\hat{h}_{j} (r) = O(r^{3+\rho})$.  By choosing certain values $(h_j(1), h'_j (1) )$ such that the coefficient $A_{+}$ and $A_{-}$ are vanishing, then we have $h_{j}(r) = O(r^{3+\rho})$. In case that $a^{+}_{j} < 3+\rho$, $h_j(r)$ has the same expression as (\ref{sol2}) with different $\hat{h}_j $, $A_+$ as follows
\begin{equation} \label{coefsol2}
\begin{split}
\hat{h}_j &= \frac{1}{a^-_j - a^+_j} \bigg( r^{a^+_j} \int_{1}^r  s^{1-a^+_j} w_j(s) ds - r^{a^-_j} \int_{1}^r  s^{1-a^-_j} w_j (s) ds\bigg),\\
A_{+} &= \frac{1}{a^-_{j}-a^+_{j}}\big( a^-_{j}  h_{j}(1) - h_{j}'(1) \big)
\end{split}
\end{equation}
and $A_-$ has the same formula as in (\ref{coefsol1}). The reason we exclude the exceptional set $A$ is that if $\rho +2 = a_j^\pm-1$, there exists some $\log$ terms appear in $\hat{h}_j$. There is only one special case remains to check; that is when $n =4$ and $\lambda''_j =0$. In this case, the solutions of the corresponding homogenous equation of (\ref{laequ1}) is generated by $1$ and $\log r$. Then, by similar computation as above, the solution $h_j$ can be written as,
\begin{align*}
h_j (r) =  h(1) + h'(1) \log r  + \log r \int_1^r s w_j (s) ds - \int_1^r \big( s \log s\big) w_j(s) ds
\end{align*}
The assumption $\rho+3 > 4-n =0$ ensures that $h_j (r) = O(r^{{\rho}+3})$. In conclution, if $\rho+3 \notin A$, then we have $h_j \in \C^{k+2,\alpha}_{\rho+3}$. It is also easy to see that $h_j = O(r^{3+\rho})$. 

To solve (\ref{laequ2}) and (\ref{laequ3}), we introduce an auxiliary functions, $D_i= -g_i''-(n-1)r^{-1} g'_i + {\lambda'_i}{r^{-2}} g_i $ and $E_i = f-g'$, then the equations (\ref{laequ2}) and (\ref{laequ3}) can be rewritten as follows,
\begin{subequations}
\begin{align}
-&E''_i - \frac{n-1}{r} E'_i +\frac{\lambda'_i+n-1}{r^2} E_i+D'_i = u_i ,\label{laequ2'} \\
&D_i - \frac{2}{r}E_i = v_i.  \label{laequ3'}
\end{align}
\end{subequations}
The system of equations can be reduces to
\begin{align}\label{laequ4}
-E''_i - \frac{n-3}{r} E'_i +\frac{\lambda'_i+n-3}{r^2} E_i = \vartheta_i,
\end{align}
where $\vartheta_i = u_i - v'_i \in C^{k-1,\alpha}_{\rho}$. The equation (\ref{laequ4}) can be solved by the same method as (\ref{laequ1}). Only to notice that the exceptional set is different from $A$,
\begin{align} \label{exset2}
B=\{b_i^{\pm}, i\geq 0 \}, \qquad b_i^{\pm} =-\frac{n-4}{2} \pm \sqrt{\Big(\frac{n-4}{2}\Big)^2 + \lambda'_i +n -3}.
\end{align}  
If $\rho +2 \notin B$, by similar discussion from (\ref{sol1}) to (\ref{coefsol2}), there exists a solution $E_i \in C^{k+1,\alpha}_{\rho+2}$ satisfying (\ref{laequ4}). It suffices to solve $g$ and $f$. The equation (\ref{laequ2'}) can be rewritten as
\begin{align} \label{laequ5}
-g''_i- \frac{n-1}{r} g'_i + \frac{\lambda'_i}{r^2} g_i = \varpi_i,
\end{align}
where $\varpi_i =2r^{-1} E_i +v_i \in \C^{k,\alpha}_{\rho+1}$. If we introduce another exceptional set, 
\begin{align}\label{exset3}
C=\{c_i^{\pm}-1, i\geq 1 \}, \qquad c_i^{\pm} =-\frac{n-2}{2} \pm \sqrt{\Big(\frac{n-2}{2}\Big)^2 + \lambda'_i },
\end{align}
then by the same method, assuming $\rho+2 \notin B\bigcup C$, there exists a solution $g_i \in \C^{k+2,\alpha}_{\rho+3}$. By the definition of $E_i$, we obtain that $f_i = g'_i + E_i \in \C^{k+1,\alpha}_{\rho+2}$. According to previous discussion, we have found coefficients satisfying the right decay condition of Fourier series of $\beta$. It remains to show that the Fourier series converges in $\C^{2}_{loc}$ topology. 

If $\lambda_j >0$, $j\in J$, we have the following estimte for $w_j (r)$, 
\begin{align}
w_j(r) =\int_{L} (\theta_1(r, y), \eta_j(y))_{g_L}  dy &= {(\lambda''_j)^{-\frac{k}{2}} }\int_L \big( \Delta_L^{\frac{k}{2}} \theta_1(r,y), \eta_j  \big)_{g_L} dy \nonumber \\
&\leq \vol(L)^{\frac{1}{2}} (\lambda''_j)^{-\frac{k}{2}} r^{\rho+1} || \theta (r,y) ||_{k,\alpha; \rho}. \label{coefest1} 
\end{align}
By by expression of $\hat{h}_j$ in (\ref{coefsol1}), we have
\begin{align} \label{coefest2}
|\hat{h}_j(r)| + r|\hat{h}_j'(r)| + r^2 |\hat{h}_j''(r)| \leq C(n, \rho) r^{\rho +3} \sup |r^{-\rho -1} w_j (r)|
\end{align}
Combining (\ref{coefest2}) with (\ref{coefest1}), we have the estimate
\begin{align}\label{coefest3}
|\hat{h}_j(r)| + r|\hat{h}_j'(r)| + r^2 |\hat{h}_j''(r)| \leq C (n, L, \rho) (1+\lambda_j)^{-\frac{k}{2}} r^{\rho+3} ||\theta(r,y)||_{k,\alpha; \rho}.
\end{align}
Applying Moser iteration to $\eta_j$,  we obtain that
\begin{align*}
||\eta_j||_{L^\infty(L)} \leq C(n,L) (\lambda''_j)^{\frac{n-1}{2}} ||\eta_j||_{L^2(L)} = C(n, L) (\lambda'')_j^{\frac{n-1}{2}}.
 \end{align*}
Then, the Schauder estimates for $\Delta_L$ implies that $||\eta_j||_{C^{2,\alpha}(L)} \leq C(n,L) (1+\lambda''_j)^{\frac{n+1}{2}}$; hence
\begin{align} \label{eigen1fest1}
|\eta_j|+r |\nabla \eta_j| + r^2 | \nabla^2 \eta_j |\leq C(n ,L)(1+ \lambda''_j) ^{\frac{n+1}{2}}r^{-1}. 
\end{align}
Recall the Weyl's law for differential forms \cite[Appendix by J. Dodziuk]{chavel1984eigenvalues}, $\lambda_j \sim j^{\frac{2}{n-1}}$. The Weyl's law together with (\ref{coefest3}) and (\ref{eigen1fest1}) can deduce that when $k \geq 2n+3 $,
\begin{align*}
|\hat{h}_j(r) \eta_j(y)| + r |\nabla (\hat{h}_j(r) \eta_j (y))| + r^2| \nabla^2 (\hat{h}_j (r)\eta_j (y))| \leq  C(n,L) j^{\big(-1 +\frac{2n-k}{n-1}\big)} r^{\rho+2} ||\theta (r,y)||_{k,\alpha; \rho+1},
\end{align*}
which implies that the series $\sum_{j\in J}\hat{ h}_j (r) \eta_j (y)$ converges in $\C^2_{loc} $. By choosing certain $h_j (1)$ and $h'_j(1)$ to ensure that $A_{+}= A_{-}=0$, $\sum_{j\in J} { h}_j (r) \eta_j (y)$ converges in $\C^2_{loc} $. To show the convergence of $\sum_{i\in I} {f}_i (r) \kappa_i (y) dr$ and $\sum_{i\in I} {g}_i (r) d_L \kappa_i (y)$, we deduce a similar estimate as (\ref{coefest3}) for $E_i$ at first. Recall that $\vartheta_i = u_i - v'_i $, then we have
\begin{align} 
\vartheta_i (r) &= \int_{L} \theta_0 (r,y) \kappa_i (y) dy - \lambda_i^{-1} \int_{L} (\partial_r \theta_1 (r, y), d_L \kappa_i (y) )_{g_L}dy \nonumber \\
&=  \lambda_i^{-\frac{k}{2}} \int_L \big(\Delta_L^{\frac{k}{2}} \theta_0 (r, y)\big) \kappa_i (y) dy - \lambda_i^{-\frac{k+1}{2}}  \int_L \big(\partial_r \big(\Delta_L^{\frac{k-1}{2}} \theta_1 (r,y) \big), d_L \kappa_i (y)\big)_{g_L} dy \nonumber \\
& \leq C(n, L) \lambda_i^{-\frac{k}{2}} r^{\rho} ||\theta ||_{k,\alpha; \rho}. \label{coefest4}
\end{align}
Recall the equation (\ref{laequ4}), together with (\ref{coefest4}), by the same method to derive (\ref{coefest3}), we have,
\begin{align} \label{coefest5}
|E_i(r)|+r|E'_i(r)|+r^2|E''_i (r)| \leq C(n, L,\rho) (\lambda'_i)^{-\frac{k}{2}} r^{\rho+2}  ||\theta||_{k,\alpha; \rho}.
\end{align}
Again, according to equation (\ref{laequ5}) and estimate (\ref{coefest5}), the same method shows that
\begin{align*}
\varpi_i (r) = \frac{2 E_i}{r} + v_i \leq  C(n, L, \rho) r^{\rho +1} (\lambda'_i)^{-\frac{k}{2}} r^{\rho+1} ||\theta||_{k,\alpha; \rho}.
\end{align*}
Therefore, we have
\begin{align} \label{coefest6}
|g_i(r)|+r|g'_i(r)|+r^2|g''_i (r)| \leq C(n, L,\rho) (1+ \lambda'_i)^{-\frac{k}{2}} r^{\rho+3}  ||\theta||_{k,\alpha; \rho}.
\end{align}
Combining (\ref{coefest5}), (\ref{coefest6}) and definition of $E_i$, we obtain,
\begin{align} \label{coefest7}
|f_i(r)|+r|f'_i(r)|+r^2|f''_i (r)| \leq C(n, L,\rho) (1+\lambda'_i)^{-\frac{k-1}{2}} r^{\rho+2}  ||\theta||_{k,\alpha; \rho}.
\end{align}
Also, we can derive the similar estimate for $\kappa_i$ and $d_L \kappa_i $ as in (\ref{eigen1fest1}), by Moser iteration and Schauder estimates 
\begin{equation} \label{eigenest2}
\begin{split}
|\kappa_i|+r |\nabla \kappa_i| + r^2 | \nabla^2 \kappa_i | &\leq C(n ,L)(1+ \lambda'_i) ^{\frac{n+1}{2}}; \\
|d_L\kappa_i|+r |\nabla d_L \kappa_i| + r^2 | \nabla^2 d_L\kappa_i | &\leq C(n ,L)(1+ \lambda'_i) ^{\frac{n+3}{2}}r^{-1}.
\end{split}
\end{equation}
Then, the estmates (\ref{coefest6}), (\ref{coefest7}) and (\ref{eigenest2}) together with  Weyl's law implies that, when $k \geq 2n+3$, the series $\sum_{i\in I} {f}_i (r) \kappa_i (y) dr$ and $\sum_{i\in I} {g}_i (r) d_L \kappa_i (y)$ converge in $C^2_{loc}$ topology. And the convergence of series also implies that 
\begin{align*}
||\beta||_{L^\infty} \leq C (n, L,\rho, k) r^{\rho +2}||\theta||_{k ,\alpha; \rho}.
\end{align*}
Since the Laplacian can be viewed as an elliptic operator for vector bundle $\wedge^1 T^*X$ over $X$, according to elliptic operator theory on conical manifolds \cite[Theorem 4.12]{marshal2002deformations}, we have Schauder-type estimate for the equation, $\Delta \beta = \theta$.
\begin{align*}
||\beta||_{{k+2,\alpha ;\rho+2}} \leq C (n, L,\rho, k) (  ||\theta||_{k, \alpha; \rho} + ||\beta||_{0; \rho+2} )\leq C  (n, L,\rho, k) ||\theta||_{k, \alpha; \rho}.
\end{align*}
\end{proof}

The proof of lemma \ref{laequ1f} not only works for 1-forms, but also for functions in weighted H\"older space. Also by spectral decompsition technique, we can solve $\Delta_0 u =f$ as follows. Recall that the increasing sequence $0= \lambda'_0 <\lambda'_1 \leq \lambda'_2 \ldots$ is eigenvalues of $\delta$ acting on functions, then the exceptional set is given by,
\begin{align}\label{excd}
D =\{ d_i^{\pm}, \text{ for } i=0,1,\ldots \}, \qquad d^{\pm}_{i} =-\frac{n-2}{2}\pm \sqrt{\Big(\frac{n-2}{2}\Big)^2 + \lambda'_i}.
\end{align}

\begin{cor}
Let $f$ be a function defined on $\overline{U}(1)$ with $f \in \C^{k,\alpha}_{\rho-2}$. If $k \geq $ and $\rho \notin D $, then there exists  a solution $u \in \C^{k+2,\alpha}_{\rho} (\overline{U}(1))$ satisfying the Laplacian equation, $\Delta_0 u = f$ and $u$ satisfies the estimate.
\begin{align} \label{laequfestc}
||u||_{\C^{k+2,\alpha}_{\rho} (U (2))}  \leq C ||f||_{\C^{k, \alpha}_{\rho-2}(U (1))}. 
\end{align} 
where $C$ only depends on $(n, L, k, \rho)$.
\end{cor}

\subsection{Laplacian on asymptotically conical manifolds} \label{aclaequss}

The asymptotic behavior of Laplacian on ALE manifolds has been studied in many references. The main idea to solve the Laplacian equation on ALE manifolds is to apply the Fredholm property of the Laplacian operator. In Lockhart \cite[Corollary 6.5]{Lockhart1981FredholmPO}, Lockhart-McOwen \cite[Theorem 1.3]{ASNSP_1985_4_12_3_409_0} and Cantor \cite[Theorem 6.3]{cantor1981elliptic}, the Fredholm property is proved in Sobolev case for the elliptic operators that are asymptotic to the Euclidean Laplacian at infinity. The H\"older case has been discussed in \cite[section 4]{chaljub1979problemes} for real dimension 3.  In Marshall \cite[Theorem 6.9]{marshal2002deformations}, the Fredholm property of Laplacian operator has been generalized to AC manifolds on both Sobolev and H\"older cases. In this section, we will summarize the key results of the H\"older cases on AC manifolds and prove the following proposition. Recall that $(C_L, g_0)$ is a Riemannian cone and $(X, g)$ asymptotic to $(C_L, g_0)$ at infinity with a diffeomrphsim $\psi: X_\infty = X-K \rightarrow (R_0, \infty) \times L$ and $ \displaystyle |\nabla^i((\psi^{-1})^* g - g_0)|_{g_0} = O(r^{-i-\tau})$, for integers $i = 0,1,\ldots, k $.

In $X_\infty$, let $\Delta_0$ and $\Delta$ be the Laplacian operators of metrics $g_0$ and $g$ respectively and $D$, the exceptional set given in (\ref{excd}). 

\begin{pro} \label{laequac} Suppose $(X, g)$ is a complete AC manifold asymptotic to a Riemannian cone $(C_L, g)$ with $\dim_{\RR} X = n \geq 4$ and $f \in \C^{k,\alpha}_{\rho-2}$. 
\begin{enumerate}
\item[\textup(i)] Let $\rho \in (0,\infty)\backslash D$, there exists a solution $u \in \C^{k+2,\alpha}_{\rho}$.
\item[\textup{(ii)}] Let $\rho \in (2-n,0)$, there exists a unique solution $u \in C^{k, \alpha}_{\rho+2}$ of $\Delta u = f$.
\item[\textup{(iii)}] Let $\rho \in (-\infty , 2-n)\backslash D $, there exists a unique solution $u = A r^{2-n} + v$, where 
\begin{align} \label{constlaequ}
A = \frac{1}{(n-2) \vol(L)} \int_{X} f d\vol_X 
\end{align}
and $v\in \C^{k+2,\alpha}_{\tilde{\rho}}$ with $\tilde{\rho} = \max \{d_1^-, \rho, 2-n-\tau\}$. 
\end{enumerate}
\end{pro} 

In the following we write $\ker (\Delta, \delta+2)=\ker(\Delta: \C^{k+2,\alpha}_{\delta+2} \rightarrow \C^{k, \alpha}_{\delta})$ and $\im (\Delta, \delta)=\im(\Delta: \C^{k+2,\alpha}_{\delta+2} \rightarrow \C^{k, \alpha}_{\delta})$. 
The key observation is that the estimate (\ref{laequfestc}) can be modified to obtain a similar estimate for elliptic operators asymptotic to $\Delta_0$ on AC Riemannian manifolds; namely, the scale broken estimate referring to \cite[Theorem 6.7]{marshal2002deformations}. Precisely, fixing $X_{2R} = \{x\in X, r(x)> 2R\}$ with $R>  R_0$,  for all $u \in \C^{k+2,\alpha}_{\rho}$, we have 
\begin{align} \label{scalbroest}
||u||_{\C^{k+2,\alpha}_{\rho}} \leq C( ||\Delta u||_{\C^{k,\alpha}_{\rho-2}} + ||u||_{\C^0 (X_{2R})}).
\end{align}
It can be proved that the Laplacian operator $\Delta: \C^{k+2,\alpha}_{\rho} \rightarrow \C^{k,\alpha}_{\rho-2}$ on $X$ is Fredholm if $\rho \notin D$. To sketch the proof of Fredholm property, we observe that the estimate (\ref{scalbroest}) implies that the set $\im (\Delta, \rho)$ is closed. To see this, let $\{f_i\}$ be a Cauchy sequence in $C^{k,\alpha}_{\rho-2}$ contained in $\im (\Delta, \rho)$. Notice that there exists a closed subspace $\mathcal{B}_\rho $ such that $\C^{k+2,\alpha}_{\rho} = \mathcal{B}_{\rho} \oplus \ker (\Delta, \rho) $ (because $\ker (\Delta, \rho)$ is finite-dimensional), then we have a bounded sequence of preimages $\{u_i\}$ in $\mathcal{B}_{\rho}$ with $\Delta u_i =f_i$. According to Arzel\`{a}-Ascoli and (\ref{scalbroest}), by taking a subsequence of $\{u_i\}$, we obtain a Cauchy sequence in $\C^0(X_{2R})$. Hence, $\{u_i\}$ is also a Cauchy sequence by passing to the subsequence. Let $f$, $u$ be the convergence of $\{f_i\}$ and $\{u_i\}$ respectively, then $f =\Delta u$ is belong to $\im(\Delta, \rho)$. Based on self-adjointness of $\Delta$, the Laplacian $\Delta: \C^{k+2,\alpha}_{\delta+2} \rightarrow \C^{k, \alpha}_{\delta}$  
\begin{align}  \label{fhdual}
 \im (\Delta, \delta)= \{f \in \C^{k,\alpha}_{\rho}, (f , h )_{L^2}=0 \text{ for all } h\in \ker ( \Delta, -n -\delta) \}
\end{align}
Then, the Fredholm property of $\Delta$ follows from the fact that $\ker(\Delta, \rho) $ is finite-dimensional.

\begin{proof}[Proof of Proposition \ref{laequac}]  By the maximal principle of harmonic functions, we can show that $\ker(\Delta, \rho)=0$ if $\rho <0$. In the cases (i) and (ii), by Fredholm property of $\Delta$ and (\ref{fhdual}), when $\rho \in (-n,-2)$, $\im (\Delta, \rho)$ is annihilated by $\ker (\Delta, -n-\rho) =0$. Hence, $\im (\Delta, \rho) = \C^{k,\alpha}_{\rho}$ for $\rho\in (-n,-2)$. The uniqueness in the case (ii) is directly from the fact $\ker (\Delta, \rho) = 0$ for $\rho \in (-n,-2)$.
.

For the case (iii), notice that $\Delta r^{2-n} = \Delta_0 r^{2-n} + (\Delta-\Delta_0) r^{2-n}= O(r^{-n-\tau})$ and
\begin{align*}
\int_{X} \Delta r^{2-n} d\vol_X &= -\lim_{r \rightarrow \infty} \int_{L(r)} \langle\nabla r^{2-n}, n \rangle_g d\vol_{L(r)}\\
& = \lim_{r \rightarrow \infty}(n-2) r^{1-n} \int_{L(r)} g(\partial r, \partial r) d\vol_{L(r)} \\
& = \lim_{r \rightarrow \infty}(n-2) \vol(L) +O(r^{-\tau}) =  (n-2) \vol(L).
\end{align*}
Consider the function $f- A \Delta r^{2-n}$, where the constant $A$ is chosen to be (\ref{constlaequ}). Observing that $f-A \Delta r^{2-n} = O(r^{\max\{-n-\tau, \rho-2\}})$ and $\int_{X} f-A\Delta r^{2-n}$  =0. Without loss of generality, we assume $\max\{-n-\tau, \rho-2\} \in (d_1^- -2, -n)$. Then (\ref{fhdual}) implies that the only annihilating functions are constant functions. The integration is vanishing indicates that $f-A \Delta r^{2-n}$ is in  $\im (\Delta, \tilde{\rho}-2)$ with $\tilde{\rho} = \max \{d^-_1, \rho, 2-n-\tau \}$; hence, we have solution $u = Ar^{2-n} +v$ with $v\in \C^{k+2,\alpha}_{\tilde{\rho}}$.
\end{proof}

\section{Proof of $dd^c$ lemma} \label{ddclems}

\subsection{Proof of Theorem \ref{ddclem} (i)} \label{pfddbar} Recall that $(X, J, g)$ is an AC K\"ahler manifold asymptotic to a Ricci-flat K\"ahler cone $(C_L, J_0, g_0)$ at infinity with a diiffeomorphism $\psi: X_\infty = X-K \rightarrow C_L -B_{R_0}$. In the end $X_\infty $, there are two sets of differential operators. In particular, we write $d$, $\partial$, $\overline{\partial}$ and $d^*$, $\partial^*$, $\overline{\partial}^*$  as the differential opertators and the dual operators with respect to $(g,J)$, which also induces the Laplacian operators $\Delta$ in the asymptotic chart. Also let $d_0$, $\partial_0$, $\overline{\partial}_0$, $d_0^*$, $\partial_0^*$, $\overline{\partial}_0^*$  and $\Delta_0$ be the Laplacian operators with respect to $(g_0,J_0)$.   Let $\omega \in \C^{k,\alpha}_{-\delta}$ be an exact real $(1,1)$-form on $X$. Then, by \cite[Theorem 3.11]{conlon2013asymptotically}, there exists a real 1-form $\alpha \in \C^{k+1,\alpha}_{1-\delta}$ such that $\omega = d\alpha$. Decompose $\alpha$ to $(1,0)$-form and $(0,1)$-form with respect to $J$, $\alpha^{1,0}$ and $\alpha^{0,1}$. Consider the following Laplacian equation for $\phi$,
\begin{align} \label{ddbareq1}
 \frac{1}{2}\Delta \phi = \overline{\partial}^* \alpha^{0,1},
\end{align}
Based on the Proposition \ref{laequac}, there exists a solution $\phi \in \C^{k+2,\alpha}_{2-\delta}$ for (\ref{ddbareq1}). Define $\xi = \alpha^{0,1}-\overline{\partial} \phi $. Due to the low decay rate, it is not necessary that $\xi =0$. The idea of the proof is to reduce the decay rate of $\xi$ and to find $\hat{\phi}$ such that $\overline{\partial} \hat{\phi} = \alpha^{0,1}$.  It is observed that 
\begin{align} \label{harcon}
\overline{\partial}^* \xi = 0, \qquad \overline{\partial} \xi =0.
\end{align}
Hence, $\xi $ is a $J$-harmonic (0,1)-form in $\C^{k+1,\alpha}_{1-\delta}$. If $1-\delta$ is nonnegative, we can reduce the order of $\xi$  as follows. Assuming that $\theta = \Delta_0 \xi = (\Delta_0 - \Delta ) \xi = O(r^{-1-\delta-\tau})$, according to lemma \ref{laequ1f}, there exists a solution $\beta = O(r^{1-\delta-\tau})$ such that $\Delta_0 \beta= \theta$ in the asymptotic chart of $X$.  Hence, we can write $\xi$ in the asymptotic chart of $X$ as 
\begin{align*}
\xi = \xi_0 + \beta, 
\end{align*}where $\xi_0$ is a harmonic 1-form  with respect to $\Delta_0$ of decay rate $1-\delta$ and $\beta$ is a 1-form of decay rate $1- \delta- \tau$. Notice that $\xi_0$ is not necessary a $(0,1)$-form with respect to $J$, but the $(1,0)$ part of $\xi_0$ has the decay rate $1-\delta-\tau$. Then, according to \cite[Lemma 2.27]{cheeger1994cone} or referring to \cite[Lemma B.1]{hein2017calabi}, the assumption that $0\leq 1-\delta <1$ implies there exists a harmonic function $\psi_0 $ with respect to $( J_0, g_0)$ such that $ d \psi_0 = \xi_0$ with $\psi_0 = \C^{k+2,\alpha}_{2-\delta}$ and $\partial \psi_0 =\C^{k+1,\alpha}_{1-\delta-\tau}$, $\overline{\partial} \psi_0 =\C^{k+1,\alpha}_{1-\delta}$. Extend $\psi_0$ to a smooth function $\tilde{\psi}_0$ on $X$. According to proposition \ref{laequac}, by solving equation $\Delta b = -\Delta \tilde{\psi}_0 $,  there exists a harmonic function $\psi$ with respect to  $(J,g)$ such that,
\begin{align*}
 \psi =\tilde{ \psi_0 }+ b.
\end{align*}
Since in asymptotic chart of $X$, $-\Delta \tilde{\psi}_0 = (\Delta_0 - \Delta) \tilde{\psi}_0 = O(r^{-\delta-\tau}) $, we have $b$ is a function of class $\C^{k+2,\alpha}_{2-\delta-\tau}$. Now, consider a new (0,1)-form given by
\begin{align} \label{reddecay}
\xi_1 = \xi - \overline{\partial} \psi,
\end{align}
which also satisfies conditon (\ref{harcon}) $\overline{\partial}\xi_1 = \overline{\partial}^* \xi_1=0$. And $\xi_1$ has the following decay rate,
\begin{align*}
\xi_1 &= \xi_0 + \beta - \overline{\partial}\psi\\
&=   \beta + \partial \tilde{\psi}_0 + \overline{\partial} \tilde{\psi}_0 - \overline{\partial} \psi\\
&= \beta + \partial \tilde{\psi}_0 -\overline{\partial} b,
\end{align*}
where $\beta$, $\partial \psi_0 \in \C^{k+1,\alpha}_{1-\delta-\tau}$. Hence, $\xi_1 \in \C^{k+1,\alpha}_{1-\delta -\tau}$. Repeating this process until the decay rate of 1-form to be negative, precisely, we can find $\hat{\xi}= \xi- \overline{\partial}\hat{\psi}$ such that $\hat{\psi} \in \C^{k+2,\alpha}_{2-\delta}$ and $\hat{\xi} \in \C^{k+1,\alpha}_{1-\delta-n\tau}$  with $1-\delta-n\tau <0$. To simplify our notation, we replace $\xi$ with the 1-form of negative decay rate. In the following proof, we assume that $\xi\in \C^{k+1,\alpha}_{1-\delta}$  ($1-\delta <0$) with $\overline{\partial}^* \xi= \overline{\partial}\xi=0$.
We introduce the following lemma.

\begin{lem} \label{decay1f} Let $k$ be a large positive integer and $1-\delta <0$, 
\begin{enumerate}
\item[(i)] Let $\xi$ is a real 1-form on $X$ of class $\C^{k+1,\alpha}_{1-\delta}$. Assuming that $\xi$ is a harmonic 1-form on $X$, then,
$\xi \in \C^{k+1,\alpha}_{-2n+2}$.
\item[(ii)] Let $\xi$ is a harmonic (0,1)-form on $X$ of class $\C^{k+1,\alpha}_{1-\delta}$ with $\overline{\partial} \xi = 0$, then $\xi\in \C^{k+1,\alpha}_{-2n+1}$  
\end{enumerate}
\end{lem}
 
\begin{proof}
Let $\xi$ be a real harmonic 1-form on $X$. According to lemma \ref{laequ1f}, $\xi$ can be represented outside a compact set of $X$ as follows,  
\begin{align}\label{decayred1f}
\xi = \xi_0 + \beta
\end{align}
where $\xi_0$ is a harmonic form with respect to $\Delta_0$ of the same decay rate as $\xi$ and $\beta $ is a 1-form of decay rate at most $1-\delta-\tau$.  By assumption that $X$ is asymptotic to the conical manifold $C_L$, $\xi_0$ can be expressed on basis of spectral decomposition with respect to Laplacian operator on link $L$,
\begin{align*} 
\xi_0 &=  \sum_{i \geq 0 } [ f_i (r) \kappa_i (y) dr + g_i (r) d_L \kappa_i(y) ] + \sum_{j\in J} h_j (r) \eta_j (y)\\
 &= \sum_{i\geq 0} [d(g_i (r) \kappa_i (y))+ E_i (r) \kappa_i (y) dr] + \sum_{j\geq 1} h_j (r) \eta_j (y). 
\end{align*}
where $E_i(r) = f_i(r) - g'_i(r)$. Then, the function $g_i$, $E_i$ satisfy the following equations, ($\dim_{\RR}X =2n$)
\begin{subequations}
\begin{align}
-&E''_i - \frac{2n-3}{r} E'_i +\frac{\lambda'_i+2n-3}{r^2} E_i= 0 ,\label{harcon1} \\
-&g_i'' - \frac{2n-1}{r} g_i' +\frac{\lambda'_i}{r^2} g_i = \frac{2}{r}E_i.  \label{harcon2}
\end{align}
\end{subequations}
By solving the equations (\ref{harcon1}) and (\ref{harcon2}), the harmonic 1-forms with respect to $\Delta_0$ of negative decay rate  at infinity have the following two types,
\begin{enumerate}
\item[(I)] Let $\lambda_i$ be all eigenvalues of $\Delta_L$ for functions on $L$,
\begin{align*}
d(r^{c_i^-} \kappa_i),\quad \text{with } c^-_i = -(n-1) - \sqrt{(n-1)^2 + \lambda'_i }
\end{align*}
\item[(II)] Let $\lambda_i$ be all \textit{positive} eigenvalues of $\Delta_L$ for functions on $L$,
\begin{align*}
 B_i r^{b_i^-} \kappa_i dr + 2 d(r^{b_i^- +1} \kappa_i ), \quad \text{with } b_i^{-} =-(n-2) - \sqrt{(n-2)^2 + \lambda'_i +2n -3},
\end{align*}
and $B_i$ is a constant given by $B_i = - (b_i^-)^2 - 2n b_i^- +\lambda_i +1-2n $.
\end{enumerate}
For the coclosed part of decomposition, consider the following equation,
\begin{align} \label{harcon3}
-h''_{j} - \frac{2n-3}{r} h'_j + \frac{\lambda_j}{ r^{2}} h_j =0, 
\end{align}
By solving (\ref{harcon3}), we have the third type of harmonic 1-form,
\begin{enumerate}
\item[(III)] Let $\lambda_j$ be the eigenvalues of $\Delta_L$ for the co-closed 1-forms on $L$,
\begin{align*}
r^{a^-_j} \eta_j, \quad \text{with } a_j^- = -(n-2)- \sqrt{ (n-2)^2  +\lambda''_j}
\end{align*}
\end{enumerate}
To estimate the eigenvalues, we notice that $L$ is a Sasakian-Einstein manifold with Ricci curvature $\Ric_{g_L} = 2(n-1) g_L $. By Lichnerowicz-Obata first eigenvalue theorem and  \cite[Lemma B.2]{hein2017calabi}, let $\lambda'_1$ and $\lambda''_1$ be the first positive eigenvalue of closed and coclosed 1-form, we have the following estimate, 
\begin{align*}
\lambda'_{1} (L)   \geq 2n-1, \qquad \lambda''_{1} (L)  \geq 4n-4.
\end{align*}
Then, we obtain the greatest possibe decay rate of harmonic 1-forms for all these three types. For type (I), the forms decays to rate at most $-2n+1$ at infinity. For type (II), the forms decays to rate at most $-2n+2$. And the forms of type (III) decays to rate at most $-2n+1$. Hence, we obtain $\xi_0 = O(r^{-2n +2})$. If $1-\delta-\tau \geq -2n+2 $, then, by expression (\ref{decayred1f}), $\xi  $ is of decay rate $1-\delta-\tau$. Again, noticing that $\beta$ is the solution of $\Delta_0 \beta = (\Delta_0 -\Delta) \xi$, there exists a $d$-closed solution $\beta$ of decay rate $1-\delta-2\tau$. Repeating the process, we can finally obtain that $\xi = O(r^{2-2n})$. 

For the second statement, let $\xi$ be a harmonic $(0,1)$-form on $X$ with $\overline{\partial} \xi = 0$. Notice that $\xi$ also can be written as in (\ref{decayred1f}), $\xi = \xi_0 + \beta$. According to (i), we can assume $\xi$, $\xi_0$ of rate $-d \leq 2-2n$ and $\beta$ of rate $-d -\tau$. $\xi$ can be decomposed to $(1,0)$ and $(0,1)$ form with respect to $J_0$,
\begin{align*}
    \xi = \xi^{1,0}_{J_0} + \xi^{0,1}_{J_0}, \quad, \xi^{1,0}_{J_0}= \frac{\xi - iJ_0 \xi}{2}, \ \xi^{0,1}_{J_0} = \frac{\xi + iJ_0 \xi}{2}.
\end{align*}
Since $\xi$ is a $(0,1)$-form with respect to $J$, $\xi- i J_0 \xi = i (J-J_0) \xi = O(r^{-d -\tau})$. The $(1,0)$ part of $\xi$ with respect to $J_0$ has rate $-d -\tau$, so does $\xi_0$. Then, the highest order terms of $\xi_0$ is a homogeneous harmonic $(0,1)$-form with respect to $J_0$, denoted by $\xi_0^{h}$. The fact $\overline{\partial}\xi=0 $ implies that $\overline{\partial}_0  \xi_0^h = O (r^{-d-1 -\tau})$, If we assume the rate of $\xi^h_0$ is greater than $1-2n$, based on the classification of harmonic 1-forms, $\xi^h_0$ must be of type (II). Then, $\xi_0^h$ can be written as,
\begin{align*}
    \xi^h_0 =  C (2b^-_i + 6 - 4n) r^{b^-_i } \kappa_i \overline{\partial}_0 r + 2 r^{b^-_i +1}\overline{\partial}_0 \kappa_i,
\end{align*}
where $i\geq 1$ and $b_i^- \geq 1-2n$. By taking differentiation, we have
\begin{align*}
    \overline{\partial}_0 \xi^h_0 = C(4b_i^- +4-4n) r^{b_i^-} \overline{\partial}_0 \kappa_i \wedge \overline{\partial}_0 r.
\end{align*}
It is easy to check that $\overline{\partial}_0 \xi^h_0$ is a nonvanishing homogeneous form of rate $-d-1 $, which contradicts against $\overline{\partial}_0  \xi_0^h = O (r^{b^-_i-1 -\tau})$. Therefore, we have $\xi = O(r^{1-2n})$.
\end{proof}

According to lemma \ref{decay1f}, we reduce the  proposition \ref{ddclem} to the case of fast decay rate. Let $\hat{\omega} = 2 \re \partial \xi $ be a closed real (1,1) form of decay rate at most $-2n+1$. Also notice that $\tr_\omega \hat{\omega}= 2 \re \overline{\partial}^* \xi =0$. Recall that there exists a natural radial function defined in $X$ and let $B_{R} = \{x\in X | r(x) \leq R\}$, for some very large positive real number $R$. Then, it is easy to check that 
\begin{align}\label{ddcint}
2 (n-2) !\int_{B_R} |\hat{\omega}|^2 d\vol =\int_{B_R} d(2\re \xi  \wedge \hat{\omega} \wedge \omega^{n-2}) = \int_{\partial B_R} 2 \re \xi \wedge \hat{\omega} \wedge \omega^{n-2}.
\end{align}
The decay conditions of $\hat{\omega}$ and $\xi$ imply that the integration in (\ref{ddcint}) tends to zero as $R$ goes to infinity. Hence, we have $\hat{\omega} =0$. In conclusion, the exact real (1,1)-form $\omega$ can be expressed as 
\begin{align*}
\omega = 2 \re \partial \alpha^{0,1} = 2 \re \partial  ( \xi + \overline{\partial} \phi) = \hat{\omega} + 2i \partial \overline{\partial} \im \phi = dd^c \im \phi,
\end{align*}
where $\phi \in \C^{k+2, \alpha}_{2-\delta}$. Therefore, we complete the proof of proposition \ref{ddclem}.

\subsection{Proof of Theorem \ref{ddclem} (ii): an obstruction of ddbar lemma}\label{ddclem'ss}  
Let $(X,J)$ be an AC K\"ahler manifold asymptotic to a Ricci-flat K\"ahler cone $(C_L, J_0)$ and $K$ be the compact set in $X$ such that $X\backslash K \cong C_L \backslash \overline{B(R_0)}$. We apply the notation as in section \ref{pfddbar}. Let $\omega $ be the real exact (1,1)-form on $X\backslash K$ with $\omega = d\alpha$, where $\alpha$ is a real 1-form on $X\backslash  K$. Then, $\alpha$ can split into the sum of a (0,1)-form and a (1,0)-form with respect to $J$, $\alpha = \alpha^{1,0} + \alpha^{0,1}$. Here $\overline{\partial}^* \alpha^{0,1}$ is a complex function defined on $X\backslash K$ of class $\C^{k,\alpha}_{-\delta}$. By extending $\overline{\partial}^* \alpha^{0,1}$ to a function $f \in \C^{k,\alpha}_{-\delta}$ on the whole manifold $X$. Then the Laplacian equation (\ref{ddbareq1}) can be rewritten as $\Delta \phi = 2 f$. According to Proposition \ref{laequac}, we have a solution $\phi \in \C^{k+2,\alpha}_{2-\delta}$. Hence, we can define a (0,1)-form on $X\backslash K$, $\xi= \alpha^{0,1}-\overline{\partial} \phi$. It is also easy to check $\overline{\partial} \xi = 0$,  $\overline{\partial}^* \xi =0 $. Based on the proof of Theorem \ref{ddclem} (i), we can reduce the rate of $\xi$ at infinity as in (\ref{reddecay}); i.e. we can find a harmonic function $\psi \in \C^{k+2,\alpha}_{2-\delta}$ such that $\hat{\xi} = \xi - \overline{\partial} \psi$ is a (0,1)-form of class $\C^{k+1,\alpha}_{1-\delta- m \tau}$ with $1- \delta -m \tau<0$ and $\overline{\partial} \hat{\xi} =0$, $\overline{\partial}^* \hat{\xi} =0 $. Then,  
\begin{align} \label{ddbar'eq1}
    \omega = 2 \re (\partial \alpha^{0,1}) = 2 \re (\partial \hat{\xi} + \partial \overline{\partial} \phi + \partial \overline{\partial} \psi) = 2 i\partial \overline{\partial} \im (\phi +\psi) + 2\re (\partial \hat{\xi}).
\end{align}
The above expression of $\omega$ (\ref{ddbar'eq1}) implies that the obstruction of ddbar lemma is the harmonic 1-form $\hat{\xi}$. Unfortunately, the obstruction term, $ \re (\partial \hat{\xi})$, cannot be vanishing in general, see counter example \ref{ctexddbar'}. However, the decay rate of the term can be reduced to enough high order. According to Lemma \ref{decay1f}, the decay rate of $\hat{\xi}$ is at most $1-2n$. Hence, $\omega = i\partial \overline{\partial} f + O(r^{-2n})$ for some real function $f \in \C^{k+2,\alpha}_{2-\delta}$, which completes the proof of theorem \ref{ddclem} (ii).

\begin{rem} \label{remobs} Fixing $X\backslash K$ as above, we assume that  $J_0 =J$ in this remark.
An interesting fact we observed in the proof of theorem \ref{ddclem} (ii) is that the obstruction of ddbar lemma is described by the class $\mathcal{H}^{0,1}_{\overline{\partial}, \overline{\partial}^*}/ \overline{\partial} \mathcal{H}$, where $\overline{\partial} \mathcal{H}$ is the $\overline{\partial}$-image of complex valued harmonic functions on $X\backslash K$ asymptotic to zero at infinity and
\begin{align*}
     \mathcal{H}^{0,1}_{\overline{\partial},\overline{\partial}^*} = \{\xi \in \Omega^{0,1} (X \backslash K) | \ \overline{\partial} \xi = 0,\ \overline{\partial}^* \xi =0 \text{ and } \xi \text{ decays to 0 at infinty }\}.
 \end{align*}
 The fact directly follows from \ref{ddbar'eq1} and the classification of harmonic 1-forms on metric cones as discussed in the proof of Lemma \ref{decay1f}. 
\end{rem}

\begin{ex} \label{ctexddbar'}
 In this example,  we discuss the ddbar lemma on $\CC^n\backslash \{0\}$. For $n\geq 3$, any exact real $(1,1)$-form $\omega$ satisfies the ddbar lemma; i.e. $\omega = dd^c f$ for a real function function $f$. In the higher dimension cases,  One refers to \cite[Lemma 5.5]{goto2012calabi} or \cite[Corollary A.3]{conlon2013asymptotically} for details. However, in the case of $n=2$, the ddbar lemma doesn't hold. According to remark \ref{remobs}, the obstruction of ddbar lemma is given by $\mathcal{H}^{0,1}_{\overline{\partial}, \overline{\partial}^*}/ \overline{\partial} \mathcal{H}$. Then, we can calculate the dimension of the spaces, $\mathcal{H}^{0,1}_{\overline{\partial}, \overline{\partial}^*}$ and $\overline{\partial} \mathcal{H}$ for each growth rate. For $k\geq 0$, we have
 \begin{align*}
     \dim_{\CC} (\mathcal{H}^{0,1}_{\overline{\partial},\overline{\partial}^*})_{-2-k} = k(k+1), \quad \dim_{\CC} (\overline{\partial} \mathcal{H})_{-2-k} = k^2.
 \end{align*}
 Let $\omega$ be an exact real $(1,1)$-form of decay rate $-\delta$, then there exists a real function $f $ of decay rate $2-\delta$ such that,
 \begin{align*}
     \omega = dd^c f + \theta_{-4} + \theta_{-5} + O(r^{-6}).
 \end{align*}
 And we can write down $\theta_{-4}$ and $\theta_{-5}$ explicitly,
 \begin{align*}
     \theta_{-4} = \re( C_{-3} \partial \xi_{-3} )= \re\bigg\{ C_{-3} \partial \bigg( \frac{\bar{z}_2 d\bar{z}_1 - \bar{z}_1  d\bar{z}_2}{(|z_1|^2+|z_2|^2)^2} \bigg)\bigg\},
 \end{align*}
 where $C_{-3}$ is a complex constant and the form $\xi_{-3}$ is a representative of $\mathcal{H}^{0,1}_{\overline{\partial}, \overline{\partial}^*}/ \overline{\partial} \mathcal{H}$ of decay rate $-3$. Similarly, $\theta_{-5}$ can be written as follows,
 \begin{align*}
     \theta_{-5} &= \re ( C_{-4,1} \xi_{-4,1} + C_{-4,2} \xi_{-4,2} ) \\
     &= \re\bigg\{ C_{-4,1} \partial \bigg( \frac{\bar{z}_1 \bar{z}_2 d\bar{z}_1 - \bar{z}^2_1  d\bar{z}_2}{(|z_1|^2+|z_2|^2)^3} \bigg) + C_{-4,2} \partial \bigg( \frac{ \bar{z}^2_2 d\bar{z}_1 - \bar{z}_1 \bar{z}_2 d\bar{z}_2}{(|z_1|^2+|z_2|^2)^3}\bigg)\bigg\}, 
 \end{align*}
 where $ C_{-4,1}$ and $C_{-4,2}$ are complex constant. In Joyce's book \cite[Theorem 8.9.2]{joyce2000compact}, the statement of ddbar lemma on $\CC^2/\{0\}$, which states that $\omega= dd^c  f + \theta_{-4}$, is wrong, as we can find infinite error terms. 
\end{ex}

\section{The Mass formula on AC K\"ahler Manifolds}
In this section, we dedicate to generalize the Hein-LeBrun's mass \cite[Theorem C]{hein2016mass} formula to AC K\"ahler manifolds asymptotic to Ricci-flat K\"ahler cones. In section \ref{ACtoLCY}, we introduce some basic set-ups on the manifolds. By introducing a system of coframes at infinity, we calculate the Chern connection forms and the curvature forms of canonical line bundle. In section \ref{massacsec},  the generalized definition of mass will be given in the  AC Riemannian manifolds. And we complete the proof of mass formula in AC K\"ahler manifolds asymptotic to Ricci-flat K\"ahler cones. 

\subsection{AC K\"ahler manifolds Asymptotic to Locally Calabi-Yau Cones} \label{ACtoLCY}

Let $(C_L, J_0)$ be a K\"ahler cone with Ricci-flat metric $g_0$. The associated link $L$ has positive Ricci curvature $\Ric(g_L) = (2n-2) g_L$. According to Myers's Theorem, the fundamental group $\pi_1$ is finite. Hence, there exists a universal covering $\widetilde{L}$ with finite deck transformation $\Gamma$ acting on $\widetilde{L}$. Let $(C_{\widetilde{L}}, \widetilde{g}_0)$ be the metric cone of $\widetilde{L}$, then $\widetilde{C}_{L}$ is a finite covering of $C_L$ with $\widetilde{C}_{{L}}/ \Gamma \cong C_L$. By pulling back the complex structure on $C_L$, the covering space $(\widetilde{C}_{L}, \widetilde{J}_0)$ is also a K\"ahler cone with Ricci-flat metric $\widetilde{g}_0$. Since $(\widetilde{C}, \widetilde{J}_0, \widetilde{g}_0)$ is K\"ahler Ricci-flat with trivial fundamental group,  then there exists a nowhere vanishing holomorphic section of canonical bundle $K_{\widetilde{C}_{L}}$, $\widetilde{\Omega}_0$ and $\tilde{\omega}_0^{n} = i^{n^2} \widetilde{\Omega}_0 \wedge \overline{\widetilde{\Omega}}_0$. $\widetilde{\Omega}_0$ induces a \textit{multi-valued} nonvanishing holomorphic section, $\Omega_0$, of $K_{C_L}$ on $C_L$ through the covering map; i.e., there exists a group representation $\kappa: \Gamma \rightarrow U(1)$ such that
\begin{align} \label{mulvalue}
    \gamma^*\widetilde{ \Omega}_0 = \kappa(\gamma) \widetilde{\Omega}_0, \qquad \gamma \in \Gamma.
\end{align}
In other words, the multi-valued section  $\Omega_0$ has $|\Gamma|$ branches and each branch differs by a constant factor $\kappa(\gamma)$, $\gamma \in \Gamma$.

Let $(C_L, J_0, g_0)$ be a Ricci-flat K\"ahler cone with the link $L$.  Let $(X, J, g)$ be an AC K\"ahler manifold asymptotic to $C_L$ together with the diffeomorphism $\Psi : X_\infty= X-K \rightarrow C_L - B_R$, where $B_R = \{x \in C_L, r(x) \leq R\}$. We also assume that $g$ decays to $g_0$ in the end $X_\infty$ with rate $-\tau = 1-n-\epsilon$ and $J$ decays to $J_0$ at infinity. 

In order to make analysis in the setting, we introduce a local frame system for K\"ahler cones. Recall that the link $L$ is a \textit{Sasakian} manifold with metric $g_L$. Let $\partial r$ be the normal vector field in radial direction, there exists a canonical vector field over $L$ defined by $\xi = J_0(r \partial r)$, namely, the \textit{Reeb vector field} on $L$. And the \textit{contact 1-form} $\eta$ is the dual of $\xi$ with respect to the metric $g_L$. The contact 1-form $\eta$ defines a subbundle of tangent bundle on the link $L$, namely $D = \ker \eta \rightarrow L$. Restricting $J_0$ on the subbundle $D$ defines an almost complex structure on $D$, which guarantees the existence of a local basis $(m_i, \overline{m}_i)$ for $D$, $i=1,\ldots, n-1$, where $m_i$ is  an (1,0) type vector with respect to $J_0|_D$. For any point in $L$, there exists a neighborhood $U$ and local basis $(\xi; m_i, \overline{m}_i)$ for the tangent bundle $T_L (U)$. Let $m_0 = r\partial r -i \xi$, then $(m_i, \overline{m}_i)_{i=0}^{n-1}$ defines a basis on $T_{C_L}(U)$.  According to the fact $r\partial r$ is a global real holomorphic vector field on $C_L$, the scaling map of $C_L$ is holomorphic. Hence, the local basis $(m_i, \overline{m}_i)_{i=0}^{n-1}$ is defined in $U \times \RR_{>0}$ by pulling forward the scaling map. Let $\{v_i, \overline{v}_i\}_{i=0}^{n-1}$ be the normalizing basis by defining $v_i = r^{-1} m_i$ and $v_i$, $\overline{v}_i$. are (1,0), (0,1) type vectors according to $J_0$. Let 1-forms, $\mu^i$, $\overline{\mu}^i$ be the dual basis of $v_i$, $\overline{v}_i$, then $J_0$ can be written explicitly as
\begin{align*}
J_0 = i \mu^i \otimes v_i - i \bar{\mu}^i \otimes \bar{v}_i.
\end{align*} 
Let $\overline{\nabla}$ be the Levi-Civita connection of $(C_L, g_0)$. One can easily check that 
\begin{align} \label{conncal1}
\overline{\nabla}_{ \partial r}v_i =\overline{\nabla}_{\partial r} \bar{v}_i=0, \qquad \overline{\nabla}_{ \partial r}\mu^i =\overline{\nabla}_{\partial r} \bar{\mu}^i=0
\end{align}
and $\overline{\nabla} J_0 =0$ also implies that $\overline{\nabla}_T v_i$ and $\overline{\nabla}_T \bar{v}_i $ are vectors of type $(1,0)$ and $(0,1)$ respectively for any tangent vector fields $T$ of $C_L$.

In the end $X_\infty$, the frames $\{\mu^i, \bar{\mu}^i\}$ and $\{v_i, \bar{v}_i\}$ are still defined by diffeomorphism $\Psi$. Precisely, there exists a finite open covering on link $L = \bigcup U_k$ such that $X_\infty \subseteq \bigcup V_k$, where $V_k = \Psi^{-1}(U_k \times \RR_{>R})$. And on each open set $V_k$, we have local frames $\{\mu^i, \bar{\mu}^i\}$ and $\{v_i, \bar{v}_i\}$. We call the data $(V_k, \{\mu^i, \bar{\mu}^i\}, \{v_i, \bar{v}_i\})$ a system of \textit{asymptotic local frames} on $X$.
A quick calculation shows that $J-J_0$ has decay rate at least $-\tau$. Precisely, if we write $\nabla$ as Levi-Civita connection of $g$, then by $J$ (resp. $J_0$) is $\nabla$-parellel (resp. $\nabla_0$-parellel), 
\begin{align} \label{cxconnre1}
 \overline{\nabla}(J- J_0) = -(\nabla-\overline{\nabla}) J.
\end{align}
If we write the difference of two Levi-Civita connection as $ A=\nabla- \overline{\nabla}$, $A$ can be viewed as a tensor on $X$. The tensor $A$ can be represented with respect to the system of asymptotic local frame $\{v_i, \bar{v}_i\}$, for instance, $(\nabla-\overline{\nabla})_{v_j}v_i = A^k_{ji} v_k + A^{\bar{k}}_{ji} \bar{v}_k $. Noticing that $|A|= O(r^{-\tau-1})$, if we also write $J$ based on a system of asymptotic local frames, then (\ref{conncal1}) implies that $|\overline{\nabla}_{\partial r} (J-J_0)|= O(r^{-\tau-1}) $. Hence, along each ray, the derivative of each coefficient of $J-J_0$ has decay rate $-\tau$. Therefore, we have $J-J_0 = O(r^{-\tau})$. Explicitly, $J$ can be written as the following tensor,
\begin{align*} 
J= J_0 + i\J^i_{\bar{j}}  \bar{\mu}^{j}& \otimes v_i +i\K^i_j  \mu^j \otimes v_i\\
&-i\overline{ \J^i_{\bar{j}}}  \mu^{j} \otimes \bar{v}_i -i\overline{ \K^i_j}  \bar{\mu}^j \otimes \bar{v}_i, 
\end{align*}
with $\J^i_{\bar{j}}$, $ \K^i_j  = O(r^{-\tau})$. According to $J^2=-1$ and comparing the terms of decay rate $-\tau$, we have
\begin{align*}
    \K^i_j =0 \quad \mod O(r^{-2\tau}).
\end{align*}
Therefore,
\begin{align}\label{cxfr1}
    J= J_0 + i\J^i_{\bar{j}}  \bar{\mu}^{j}& \otimes v_i - i\overline{ \J^i_{\bar{j}}}  \mu^{j} \otimes \bar{v}_i + O(r^{-2\tau})
\end{align}

According to (\ref{cxconnre1}), we can derive the relationship between $A$ and $\overline{\nabla} \J$,
\begin{align*}
    -(\nabla- \overline{\nabla})J_0 = & -2i A^{\bar{j}}_{i\bar{k}}\mu^i \otimes \bar{\mu}^k \otimes \bar{v}_j - 2i A^{\bar{j}}_{ik} \mu^i \otimes {\mu}^k \otimes \bar{v}_j\\
    & +2i A^{j}_{\bar{i}\bar{k}} \bar{\mu}^i \otimes \bar{\mu}^k \otimes {v}_j + 2i A^{j}_{\bar{i} k} \bar{\mu}^i \otimes {\mu}^k \otimes {v}_j
\end{align*}
and 
\begin{align*}
    \overline{\nabla}(J-J_0) = i(\overline{\nabla}\J)^j_{\bar{k}, i} + i (\overline{\nabla}\J)^j_{\bar{k}, \bar{i}} - i (\overline{\nabla}{\J})^{\bar{j}}_{k, i}- i(\overline{\nabla}{\J})^{\bar{j}}_{k, \bar{i}}.
\end{align*}
The relation $\overline{\nabla}(J-J_0) =-(\nabla- \overline{\nabla})J_0 +O(r^{-2\tau-1})$ implies that,
\begin{align}\label{cxconnre2}
    (\overline{\nabla}\J)^j_{\bar{k}, i} = 2 A^{j}_{i\bar{k}} + O(r^{-2\tau-1}), \qquad 
    (\overline{\nabla}\J)^j_{\bar{k}, \bar{i}} = 2 A^{j}_{\bar{i}\bar{k}} + O(r^{-2\tau-1})
\end{align}
The fact that $A$ is symmetric implies that   
\begin{align} \label{symoften}
    (\overline{\nabla}\J)^j_{\bar{k}, \bar{i}}- (\overline{\nabla}\J)^j_{\bar{i}, \bar{k}} = O(r^{-2\tau-1}) 
\end{align}
Recall that, on the cone $C_L$, there exists a nowhere vanishing multi-valued holomorphic volume form $\Omega_0$. By restricting to local charts $V_k$, $\Omega_0$ can be expressed as a single-valued form for each $\gamma \in \Gamma$ with respect to the asymptotic local frame of $V_k$,
\begin{align*}
\big(\Omega_0 |_{V_k} \big)_\gamma = \kappa(\gamma) f \mu^0 \wedge \cdots \wedge \mu^{n-1},
\end{align*}
where $\kappa(\gamma)$ is defined in (\ref{mulvalue}). And the condition $\overline{\nabla}_{\partial r} \Omega_0 =0  $, together with (\ref{conncal1}), implies that $f$ is bounded on $V_k$. It is possible to define a smooth $(n, 0)$ form $\Omega$ on $X_\infty$ with respect to $J$ which tends to $\Omega_0$ at infinity with decay rate $-\tau$. Let $\varpi^i =  (\mu^i - iJ \mu^i)/2$ be the $(1,0)$ part of $\mu^i$ with respect to $J$,
\begin{align}
\big(\Omega|_{V_k}\big)_\gamma &= \kappa(\gamma) f \varpi^0 \wedge \ldots \wedge \varpi^{n-1} \nonumber \\
&=\big(\Omega_0|_{V_k}\big)_\gamma +  \frac{\kappa(\gamma)}{2}\sum_{i=0}^{n-1}(-1)^{i-1}  f \J^i_{\bar{j}} \bar{\mu}^j \wedge \mu^0 \wedge \ldots \wedge \widehat{\mu^i} \wedge \ldots \wedge \mu^{n-1} + O(r^{-2\tau}) \label{n0expr1}
\end{align}
It's easy to check that $(\Omega|_{V_k}, V_k )$ can be patched together to obtain a multi-valued form $\Omega$ on the end, as $\Omega|_{V_k}$ can be viewed as $(n,0)$ projection of $\Omega_0$ in each $(V_k, J)$. The fall-off condition of $\Omega$ to $\Omega_0$ indicates that $\Omega$ is nowhere vanishing if the radial function $r$ large enough. Without loss of generality, we assume that $\Omega$ is nowhere vanishing on $X_\infty$. We can rewrite the formula (\ref{n0expr1}) on $V_k$ as follows,
\begin{align}\label{n0expr2}
\Omega = \Omega_0 + i \langle J-J_0 , \Omega_0 \rangle + O(r^{-2\tau}). 
\end{align}
where $\langle J-J_0 , \Omega_0 \rangle$ means the contraction of tensor fields by viewing $J-J_0$ and $\Omega_0$ as tensor fields on $V_k$. Let $h$ be the standard metric on canonical bundle of $X$ defined to be $h(\mathcal{S})= |\mathcal{S}\wedge \overline{\mathcal{S}}|/ \omega^n $ for any section $\mathcal{S}$ of canonical bundle on $X$. Let $D_h$ be the Chern connection of $h$, then there exists a connection 1-form $\theta$ defined on $X_\infty$ as $D_h \mathcal{S} = \theta \otimes \mathcal{S}$. Locally, we can take $\mathcal{S} = \Omega_\gamma$, where $\Omega_\gamma$ is a single-valued local section of canonical bundle. Recall that $(0,1)$ part of the covariant derivative $D^{0,1}_h =\bar{\partial}$, then $D^{(1,0)} \Omega_\gamma =\bar{\partial} \Omega_\gamma = d\Omega_\gamma$. By writing $D^{0,1}_h \Omega_\gamma = \alpha \otimes \Omega_\gamma$, in the following lemma, we derive an explicit formula of $\alpha$, in which we can see that $\alpha$ is independent of the choice of $\gamma$ modeling the higher decay terms. 
\begin{lem} \label{01conn}
Let $\alpha$ be a $(0,1)$ part of connection 1-form satisfying  $D_h^{0,1} \Omega_\gamma = \alpha \otimes \Omega_\gamma$, then in each asymptotic local frame $\{V_k, \mu^j \}$, 
\begin{align*}
    \alpha|_{V_k} = - \sum_{i,j} A^{i}_{i\bar{j}} \bar{\mu}^j  + O(r^{-2\tau-1})
\end{align*}
\end{lem}

\begin{proof}
 Observe that the operator $d$ can be represented as $d= \mu^i \wedge \overline{\nabla}_{v_i} + \bar{\mu}^i \wedge \overline{\nabla}_{\overline{v}_i} $. Then,  $\Omega_0$ is a holomorphic volume form on $(C_L, J_0)$ implies that
\begin{align*}
0 = d \Omega_0 = \bar{\mu}^i \wedge \overline{\nabla}_{\bar{v}_i} \Omega_0.
\end{align*}
Hence, $\overline{\nabla}_{\bar{v}_i} \Omega_0 =0$. Together with (\ref{n0expr2}), we have 
\begin{align}
d \Omega &=d\Omega_0 + i d \langle J- J_0, \Omega_0  \rangle + O(r^{-2\tau-1})\nonumber\\
& = \frac{i}{2} \bar{\mu}^i \wedge \langle \overline{\nabla}_{\bar{v}_i} (J-J_0), \Omega_0 \rangle  \label{d10part1} \\
& \qquad + \frac{i}{2} \mu^i \wedge \big( \langle \overline{\nabla}_{v_i} (J-J_0), \Omega_0 \rangle + \langle  J-J_0, \overline{\nabla}_{v_i}\Omega_0 \rangle \big) +O(r^{-2\tau-1}). \nonumber
\end{align}
To simplify the expression in (\ref{d10part1}), we apply (\ref{cxconnre1}) or  (\ref{cxconnre2}) and (\ref{symoften}). Then, we have
\begin{equation}\label{d10part2}
\begin{split}
i \bar{\mu}^i \wedge \langle \overline{\nabla}_{\bar{v}_i} (J-J_0), \Omega_0 \rangle &= i\bar{\mu}^i \wedge \langle  (\nabla- \overline{\nabla})_{\bar{v}_i} J_0 , \Omega_0   \rangle + O(r^{-2\tau-1}) \\
&= \kappa(\gamma) \sum_{i,j,k}(-1)^i 2f A_{\bar{i} \bar{j}}^{k} \bar{\mu}^i \wedge \bar{\mu}^j \wedge \mu^0 \wedge \ldots \wedge \widehat{\mu^k} \wedge \ldots \wedge \mu^n + O(r^{-2\tau-1})  \\
&=O(r^{-2\tau-1}). 
\end{split}
\end{equation}
Let $\overline{\Gamma}$ be the Christoffel symbol of $ \overline{\nabla}$, then a quick calculation shows that
\begin{align} \label{d10part3} 
i\mu^i \wedge \langle J-J_0, \overline{\nabla}_{v_i} \Omega_0  \rangle = \kappa(\gamma) \sum_{i,j,k} \K^i_{\bar{j}} (f_i - f \overline{\Gamma}^k_{ik}) \bar{\mu}^{{j}} \wedge \mu^0 \wedge \ldots \wedge \mu^{n-1},
\end{align}
and
\begin{align} 
i\mu^i \wedge \langle \overline{\nabla}_{v_i} (J-J_0),  \Omega_0  \rangle &= i \mu^i \wedge \langle (\nabla - \overline{\nabla})_{v_i} J_0, \Omega_0  \rangle + O(r^{-2 \tau -1 }) \nonumber \\
&=-2 \kappa(\gamma) \sum_{i,j} f A_{i \bar{j}}^i \bar{\mu}^j \wedge \mu^0 \wedge \ldots \wedge \mu^{n-1} + O(r^{-2 \tau-1}). \label{d10part4} 
\end{align}
Then, (\ref{d10part2})-(\ref{d10part4}) imply that $D^{0,1}_h \Omega = \alpha \otimes \Omega$ with 
\begin{align}\label{d10part5}
\alpha|_{V_k} = \frac{1}{2} \sum_{i,j,k} \big( \K^i_{\bar{j}} (\log f)_i - \K^i_{\bar{j}} \overline{\Gamma}^{k}_{ik}- 2 A^i_{i \bar{j}} \big)\bar{\varpi}^j + O (r^{-2\tau-1})
\end{align}
The formula (\ref{d10part5}) can be simplified further by considering the holomorphic volume form condition $\nabla_{\bar{v}_i}\Omega_0=0$,
\begin{align} \label{d10partcon}
    \nabla_{\bar{v}_i}\Omega_0 = \kappa(\gamma) ((\log f)_{\bar{i}} - \sum_{k}\overline{\Gamma}^{k}_{\bar{i}k}) f \mu^1 \wedge \ldots \wedge \mu^n  
\end{align}
The base K\"ahler metric $\omega_0 = g_0 (J_0\cdot, \cdot)$ can be written as $\omega_0 = i (g_0)_{i\bar{j}} \mu^i \wedge \bar{\mu}^j$. The relation $i^{n^2}\Omega_0 \wedge \overline{\Omega}_0 =\frac{1}{n!} \omega_0^n $ implies that $|f|^2 = \det((g_0)_{i\bar{j}})$, then
\begin{align*}
    (\log|f|^2 )_i & = (g_0)^{j\bar{k}}  (g_0)_{j\bar{k},i} \\
    & = (g_0)^{j\bar{k}}\big(\overline{\Gamma}^{p}_{ij}(g_0)_{p\bar{k}}+ \overline{\Gamma}^{\bar{q}}_{i\bar{k}}(g_0)_{j\bar{q}} \big) \\
    & = \sum_j \overline{\Gamma}^{j}_{ij} + \sum_{k}\overline{\Gamma}^{\bar{k}}_{i\bar{k}}.
\end{align*}
Therefore, combining with (\ref{d10partcon}),
\begin{align*}
    (\log f)_i - \sum_k \overline{\Gamma}^k_{ik} & = (\log |f|^2)_i - (\log \bar{f})_i - \sum_k \overline{\Gamma}^k_{ik} \\
    & = (\log |f|^2)_i - \sum_{k} \overline{\Gamma}^{\bar{k}}_{i\bar{k}} - \sum_k \overline{\Gamma}^k_{ik} =0
\end{align*}
we complete the proof of the lemma.
\end{proof}

Therefore, after fixing a Ricci-flat K\"ahler cone $(C_L, J_0, g_0)$, (\ref{cxconnre2}) implies that $\alpha$ only depends on the complex structure $J$ after modeling higher decay terms. By assuming that $D^{1,0} \Omega = \beta \otimes \Omega$, we have 
\begin{align*}
\partial h(\Omega) &= \langle D^{1,0} \Omega, \Omega\rangle + \langle \Omega, \overline{D^{0,1} \Omega} \rangle \\
&= h(\Omega) \beta +h(\Omega) \overline{\alpha}.
\end{align*}
Hence, we have $\beta = -\overline{\alpha} + \partial \log h(\Omega)$. Let $\theta_\omega$ be the real part of $-i\theta$, we obtain the following expression,
\begin{align}\label{acconnform1}
\theta_\omega = 2 \im \alpha -\frac{1}{2} d^C \log \frac{\omega^n}{|\Omega_0 \wedge \bar{\Omega}_0|} + O(r^{-2\tau-1}).
\end{align}
Recall that, in $X_\infty$, the Ricci form   is given by $\rho = d \theta_\omega$ and the first Chern class is $c_1 = [\rho/\pi]$. By adding a copy of link $L$ to the end of $X$, we have the following long exact sequence,
\begin{align} \label{lexseq1}
\cdots \rightarrow H_{dR}^1(L) \rightarrow H^2_c (X) \xrightarrow{\iota}  H_{dR}^2 (X) \rightarrow H_{dR}^2 (L)  \rightarrow \cdots 
\end{align}
Noting that $L$ is the link of a Ricci-flat K\"ahler cone $C_L$, It is known that $L$ is Sasakian-Einstein with $Ric_{g_L} = 2(n-1) g_L$. According to the standard Bochner technique, we have $H^1_{dR} (X)$ vanishes. 
\begin{lem} \label{cptclasslm}
Let $ [\varrho] \in H^2_{dR} (X)$ where $\varrho$ is a closed 2-form with decay rate $-\delta$, $\delta> 2$. Then, there exists a unique preimage of $[\varrho]$ in $H^2_c(X)$, denoted by $\iota^{-1} [\varrho]$.
\end{lem}

\begin{proof} The uniqueness of preimage directly follows from $H^{1}_{dR} (L) =0 $. Now, we describe the mapping $r: H^{2}_{dR} (X) \rightarrow H^2_{dR} (L)$ in (\ref{lexseq1}).  Notice that, topologically, $X_\infty\cong L \times (R_0, \infty)$, then, the closed 2-form $\varrho$ can be written as $\varrho = \eta_2 + dr \wedge \eta_1$.  According to the fall-off condition of $\varrho$ with decay rate $-\delta$, the norm of $\eta_1$, $\eta_2$ on $L$ satisfy 
\begin{align*}
|\eta_1|_{g_L} = r| \eta_1|_{g_0} = O(r^{1-\delta}), \qquad |\eta_2|_{g_L}= r^2 |\eta_2|_{g_0} = O (r^{2-\delta})
\end{align*}
By adding a copy of link $L$ at infinity, then in topological space $\overline{X}_\infty = L \times (R_0, \infty]$, $\eta_1$ and $\eta_2$ can extend to infinity by defining zero forms at infinity. Hence, we obtain an extension of $\varrho$ in the space $\overline{X}_\infty$, denoted by $\overline{\varrho}$, with $\overline{\varrho} |_{\infty} = 0$.  Then, the class $[\varrho]$ under the mapping $r$ is given as follows,
\begin{align*}
r([\varrho]) = \big[ \overline{ \varrho} |_{\infty} \big]  =0
\end{align*}
Then, the long exact sequence (\ref{lexseq1}) implies that we have a preimage of $[\varrho]$ in $H_c^2(X)$.
\end{proof}
Based on the proof of lemma \ref{cptclasslm}, we can explicitly construct a closed form with compact support represents $\iota^{-1} [\varrho]$. Recall that $\varrho = \eta_2 + dr \wedge \eta_1$ in $X_\infty$. We integrate the form $dr \wedge \eta_1$ in $r$ direction; in particular $\tilde{\eta}_1 (r) = - \int_{\infty}^r dr \wedge \eta_1$. The 1-form $\tilde{\eta}_1$ is well-defined as $|\eta_1|_{g_L}$ has decay rate $1-\delta < -1$. Then $\varrho= \eta_2 - d_L  \tilde{\eta}_1  + d \tilde{\eta}_1 (r)$. Let $\tilde{\eta}_2 = \eta_2-d_L \tilde{\eta}_1$, then the closedness of $\tilde{\eta}_2$ implies that 
\begin{align*}
d \tilde{\eta}_2 = d_L \tilde{\eta}_2 + dr \wedge  \partial_r \tilde{\eta}_2 =0 \quad \Rightarrow \quad  d_L \tilde{\eta}_2 =0, \quad \partial_r \tilde{\eta}_2 =0. 
\end{align*}
Observing that $ \lim_{r\rightarrow \infty}\tilde{\eta}_2 =0$, together with $\partial_r \tilde{\eta}_2 =0$, we obtain that $\tilde{\eta}_2 = 0$. Hence, in $X_\infty$, $\varrho = d \tilde{\eta}_1$. Let $\chi$ be a smooth cut-off function defined in $X_\infty$ such that
\begin{align} \label{cutoff}
\begin{split}
f(x) =\begin{cases}
1, \quad \text{ if } |x| \geq 3 R_0;\\
0, \quad \text{ if } |x| \leq 2 R_0.
\end{cases}
\end{split}
\end{align}
Then, $\varrho-d (\chi \tilde{\eta}_1)$ represents the class $\iota^{-1} [\varrho]$. Back to the case of the first Chern class $c_1 = [\rho /2\pi]$, the expresstion (\ref{acconnform1}) shows that the Ricci form decays at infinity with order $-\tau-2 < -2$. By lemma \ref{cptclasslm}, we have $\iota^{-1} c_1 \in H^2_c (X)$ and $(\rho- d(f \theta_\omega))/2\pi$ represents this class. 

\subsection{The Mass Formula} \label{massacsec}

To give a reasonable definition of mass on AC K\"ahler manifolds, we will invoke an expression of mass without applying Euclidean coordinates. As mentioned in \cite[Section 3.1.3]{lee2019geometric}, the ADM mass can be interpreted as the integral of linearization of scalar curvature at the Euclidean background metric. This idea can be generalized to the AC K\"ahler cases if we replace the background Euclidean metric with the Riemannian cone metric. If we perturb the cone metric $g_0$ and consider $g(t) = (1-t) g_0 + t g_1$ where $g_1$ is another Riemannian metric defined on $C_L$ with the corresponding scalar curvature $R(t) \in L^1 (C_L)$ and bounded near $0$, according to Appendix A, lemma  \ref{linofs1} for any real local frame on $X_\infty$
\begin{align*}
    \frac{d}{dt}\Big|_{t=0} R(t) = \overline{\nabla}^i \overline{\nabla}^j g_{ij} - \overline{\Delta} \tr_{g_0} g
\end{align*}
Notice that 
\begin{align*}
    \int_{C_L} \frac{d}{dt}\Big|_{t=0} R(t) = \lim_{r \rightarrow \infty} \int_{L(r)} (\overline{\nabla}^j g_{ij}- (\tr_{g_0}g)_i) n^i d\vol_{L(r)},
\end{align*}
where $n$ is a outer normal vector field of $L(r)$ and the integrand of the right-hand side is independent of the choice of real frame. 
Let $(X, g)$ be an AC Riemannian manifold asymptotic to metric cone $(C_L, g_0)$, we give the abstract definition of mass on $(X,g)$ by
\begin{align} \label{massdef}
    \m(g) =  \frac{1}{2(2n-1) \vol(L)}\lim_{r \rightarrow \infty} \int_{L(r)} (\overline{\nabla}^j g_{ij}- (\tr_{g_0}g)_i) n^i d\vol_{L(r)}
\end{align}
 Let $(X, g, J)$ be an AC K\"ahler manifold asymptotic to a Ricci-flat K\"ahler cone $(C_L, g_0, J_0)$ of complex dimension $n$ with the metric $g$ asymptotic to $g_0$ in the end $X_\infty$, the remaining part of this section is aimed to derive a mass formula on $(X, g, J)$ which generalize the Hein-LeBrun's mass formula \cite{hein2016mass}. Observe that
 \begin{align} \label{massfp1}
     -\frac{1}{2} (\tr_{g_0} g)_i = - \frac{1}{2} \Big(\log \frac{\det g}{\det g_0} \Big)_i + O(r^{-2\tau-1})
 \end{align}
 and according to the calculation in Appendix (\ref{fofa})
 \begin{align} \label{massfp2}
     \overline{\nabla}^j g_{ij} -\frac{1}{2} (\tr_{g_0}g)_i = \frac{1}{2} g_0^{jk} ( \overline{\nabla}_k g_{ij} +  \overline{\nabla}_j g_{ik}-  \overline{\nabla}_i g_{jk}) = g_0^{jk} A_{jk, i},
 \end{align}
where $A_{jk,i}=g_{il} A^l_{jk}$. Also observe that the formulas (\ref{massfp1}) and (\ref{massfp2}) can relate with the terms of $\theta_\omega$ in (\ref{acconnform1}). Then, we have the following mass formula,

\begin{pro} \label{massf1} Let $(X, g, J)$ be an AC K\"ahler manifold asymptotic to a locally Calabi-Yau cone $(C_L, g_0, J_0)$ of complex dimension $n$ with $|\overline{\nabla}^i (g-g_0)|_{g_0} = O(r^{-\tau-i})$ for $i = 0,1,\ldots$, where $\tau = 1-n-\epsilon$. Let $\omega$ represent the K\"ahler form of $g$ on $X$. Let $\theta_\omega$ be the connection 1-form as in (\ref{acconnform1}). Then, the mass has the following formula,
\begin{align} \label{massf2}
\m( g) = \frac{1}{(2n-1) (n-1)! \vol(L)} \lim_{r \rightarrow \infty} \int_{L(r)} \theta_\omega \wedge \omega^{n-1}.
\end{align}
\end{pro}

\begin{proof} According to the observation from  (\ref{massfp1}) and (\ref{massfp2}), the mass formula can be rewritten on asymptotic local coframe $\{\mu_i, \bar{\mu}_i \}$, 
\begin{align*}
    \m(g) =  \frac{1}{2(2n-1) \vol(L)}\lim_{r\rightarrow \infty} \bigg\{ \int_{L(r)} * 2 \re \big( A_{, \bar{i}} \bar{\mu}^i\big)   - \int_{L(r)} * \bigg(d \log \frac{\sqrt{det g}}{\sqrt{\det{g_0}}}\bigg) + O(r^{-2\epsilon})\bigg\}
\end{align*}
where $*$ is the Hodge star operator of on $(X,g)$ and $A_{,i} = 2g_0^{j\bar{k}}A_{j\bar{k},i}$. Also notice that 
\begin{align*}
 * \bigg(d \log \frac{\sqrt{\det g}}{\sqrt{\det{g_0}}}\bigg) =  -\Big(Jd \log \frac{\omega^n}{ |\Omega_0 \wedge \Omega_0|} \Big)\wedge \frac{\omega^{n-1}}{(n-1)!} = d^C \log \frac{\omega^n}{|\Omega_0 \wedge \Omega_0|} \wedge \frac{\omega^{n-1}}{(n-1)!}.
\end{align*}
and 
\begin{align*}
     \im \alpha  \wedge \frac{\omega^{n-1}}{(n-1)!} = (J \re \alpha) \wedge \frac{\omega^{n-1}}{(n-1)!} = * (\re \alpha )
\end{align*}
Comparing with the explicit expression of $\theta$ in (\ref{acconnform1}), then the mass formula can be rewritten as
\begin{align} \label{massf2}
    \m(g) = \frac{1}{(2n-1) \vol(L)} \lim_{r\rightarrow \infty} \bigg\{ \int_{L(r)} * \re\big( A_{, \bar{i}} \mu^{\bar{i}}- 2\alpha \big) + \frac{1}{(n-1)!} \int_{L(r)} \theta_\omega \wedge \omega^{n-1} \bigg\}
\end{align}
It suffices to prove the difference $A_{,\bar{i}}\mu^{\bar{i}}- 2\alpha$ does not contribute to the integral. If we write the metric $g$ in terms of frames $\{\mu^i, \mu^{\bar{i}}\}$, 
\begin{align*}
    g = g_{ij} \mu^i \otimes \mu^j + g_{i \bar{j} } \mu^i \otimes \mu^{\bar{j}} + \overline{g_{i \bar{j}}} \mu^{\bar{i}} \otimes \mu^j + \overline{g_{ij}} \mu^{\bar{i}} \otimes \mu^{\bar{j}}.
 \end{align*}
Noting that $\omega = g(J\cdot, \cdot)$ and inserting (\ref{cxfr1}) into the expression of $g$, we obtain that,
\begin{align*}
\omega = i g_{i\bar{j}} \mu^i \wedge \mu^{\bar{j}} + i( g_{i\bar{k}} \J^{\bar{k}}_j - g_{j \bar{k}} \J^{\bar{k}}_i ) \mu^i \wedge \mu^j - i ( g_{k \bar{i}} \J^{k}_{\bar{j}} - g_{k \bar{j}} \J^k_{\bar{i}}) \mu^{\bar{i}} \wedge \mu^{\bar{j}} + O(r^{-2\tau}).
\end{align*}
The (0,2)-part of $\omega$ with respect to $J_0$ is given by $\omega^{0,2} = -i(g_{k\bar{i}} \J^k_{\bar{j}}- g_{k\bar{j}}\J^k_{\bar{i}}) \mu^{\bar{i}}\wedge \mu^{\bar{j}} $.Now we calculate $d^* \omega^{0,2}$. Notice that $d^* = [d^C, \Lambda]$, then
\begin{equation} 
\begin{split} \label{massfp3}
    d^* \omega^{0,2} & = i \Lambda (J\circ d (g_{k\bar{i}}\J^k_{\bar{j}} - g_{k \bar{j}} \J^k_{\bar{i}})\mu^{\bar{i}}\wedge \mu^{\bar{j}} ) + O(r^{-2\tau-1})\\
    & = (\overline{\nabla} \J)^k_{\bar{j}, l}(g_0)_{k\bar{i}} (g_0)^{l \bar{i}} \mu^{\bar{j}} - (\overline{\nabla} \J)^k_{\bar{i}, l}(g_0)_{k \bar{j}} (g_0)^{l\bar{i}} \mu^{\bar{j}} + O(r^{-2\tau-1}) \\
    & = 2 A^{l}_{l \bar{j}} \mu^{\bar{j}} -  2 (g_0)^{i \bar{k}}A_{ i \bar{k}, \bar{j}} \mu^{\bar{j}} + O(r^{-2\tau-1})\\
    & =  2\alpha - A_{, \bar{j}} \mu^{\bar{j}}+ O(r^{-2\tau-1}).
\end{split}
\end{equation}
Inserting (\ref{massfp3}) to the first term of (\ref{massf2}), we have
\begin{align*}
    \int_{L(r)} * \re \big(A_{, \bar{i}} \mu^{\bar{i}} -2 \alpha \big) & = \frac{1}{2}\int_{L(r)} * \big( d^* \omega^{0,2} + d^* \overline{\omega^{0,2}} \big) + O(r^{-2 \epsilon})\\
    & =\frac{1}{2} \int_{L(r)} d * (\omega^{0,2}+ \overline{\omega^{0,2}})+ O(r^{-2 \epsilon}) = O (r^{-2\epsilon}),
\end{align*}
hence, the proof is completed.
\end{proof} 

\begin{thm}\label{massf2}
Let $(X, g, J)$ satisfy the same condition as in  proposition \ref{massf1}, then we have the following mass formula
\begin{align} \label{massf3}
\m( g) = -\frac{2\pi \langle \iota^{-1} c_1, [\omega]^{n-1} \rangle}{(2n-1) (n-1)! \vol(L)} + \frac{1}{2(2n-1) \vol(L)} \int_{X} R_g d\vol_g
\end{align}
where $c_1$ is the first Chern class of $(X, J)$, $[\omega]$ is the K\"ahler class of $g$, $R_g$ is the scalar curvature of $g$ and $(\cdot ,\cdot ) $ is the duality pairing between $H^{2}_c (X)$ and $H^{2n-2} (X)$.
\end{thm}

\begin{proof}
In proposition \ref{massf1}, we have already proved a mass formula (\ref{massf2}), where $\theta_\omega$ is a 1-form defined on $X_\infty$ with $d\theta_\omega = \rho$. Consider a smooth cut-off function $f$ as in (\ref{cutoff}) which is $\equiv 0$ on $X- X_\infty$ and $\equiv 1$ if the radius is of large enough. According to the discussion after lemma \ref{cptclasslm}, the first Chern class admits a representative with compact support $\iota^{-1} c_1 = [\rho_0 / 2\pi]$, where $\rho_0 = \rho - d(f \theta_\omega)$ is compactly supported. Let $X_{r\leq R} =\{x\in X, r(x) \leq R \}$ with $R > 3 R_0$, then $\rho_0$ has compact support on $X_{r < R}$. We have
\begin{align*}
    \frac{(n-1)!}{2} \int_{X_{r \leq R}} R_g d\vol_g & = \int_{ X_{r \leq R}} \rho \wedge \omega^{n-1} \\
    & = \int_{X_{r \leq R}} \big(\rho_0 + d(f \theta_{\omega})\big) \wedge \omega^{n-1}\\
    & = {2\pi} \langle \iota^{-1} c_1, [\omega]^{n-1}  \rangle +  \int_{L(R)} \theta_{\omega} \wedge \omega^{n-1} 
\end{align*}
Then, according to (\ref{massf2}), we have
\begin{align*}
    \m(g) & = \frac{1}{(2n-1) (n-1)! \vol(L)} \lim_{r\rightarrow \infty} \int_{L(r)} \theta_\omega \wedge \omega^{n-1} \\
    & = -\frac{2\pi \langle \iota^{-1} c_1, [\omega]^{n-1} \rangle}{(2n-1) (n-1)! \vol(L)} + \frac{1}{2(2n-1) \vol(L)} \int_{X} R_g d\vol_g
\end{align*}
which completes the proof of mass formula
\end{proof}
\section{Expansion of AC K\"ahler Metrics } \label{expanacsec}
\subsection{Proof of Theorem \ref{expanthm} (i)}
According to the conditions given in theorem \ref{expanthm} (i), $\omega_1$ and $\omega_2$ are two K\"ahler forms in the same K\"ahler class and decay to the model K\"ahler form $\omega_0$ with rate $-\tau$ ($\tau = n-1 +\epsilon$) in $X_\infty$. We also assume that their scalar curvature are identically equal. Firstly, we obtain an expression of scalar curvature $R(\omega)$ with respect to some K\"ahler metric $\omega$ satisfying the fall-off condition, based on the formula (\ref{acconnform1}). Noting that, in (\ref{acconnform1}), the leading term of $ \im \alpha$ only depends on the complex structure $J$ by lemma \ref{01conn}.
In particular, let $A_J  = 2 \im \alpha+ O(r^{-2\tau-1})$  in (\ref{acconnform1}), then
\begin{align} \label{scalarcurv}
R(\omega) = \tr_\omega \Big(dA_J - \frac{1}{2}dd^C \log \frac{\omega^n}{|\Omega_0 \wedge \overline{\Omega}_0|} \Big) + O(r^{-2-2\tau})
\end{align}
Theorem \ref{ddclem} implies that there exists a function $\varphi\in \C^{k+2,\alpha}_{2-2\tau}$ such that $\omega_2 = \omega_1 +dd^C \varphi$. Based on the condition $R_1 = R_2$, we have
\begin{align}
 0 &= \tr_{\omega_2} \rho (\omega_2) - \tr_{\omega_1} \rho(\omega_1) \nonumber\\
& = \tr_{\omega_1} (\rho({\omega_2 })-\rho({\omega_1})) + (\tr_{\omega_2}-\tr_{\omega_1} ) \rho({\omega_2}) \label{difsc}
\end{align}
By inserting (\ref{acconnform1}), we rewrite the first term of (\ref{difsc}),
\begin{align*}
\tr_{\omega_1} (\rho({\omega_2 })-\rho({\omega_1}))&=\tr_{\omega_1}\bigg(dA_J-\frac{1}{2} dd^C \log \frac{(\omega_1+dd^C \varphi)^n}{|\Omega \wedge \overline{ \Omega} | }-dA_J+ \frac{1}{2} dd^C \log \frac{\omega_1^n}{|\Omega \wedge \overline{ \Omega} | }\bigg) \\
&=-\frac{1}{2}  (\tr_{\omega_1} dd^C)^2 \varphi + O(r^{-2\tau-2})\\
& =-\frac{1}{2} \Delta_{\omega_1}^2 \varphi + O(r^{-2\tau -2}).
\end{align*}
By writing the second term of (\ref{difsc}) in locally asymptotic frame, it is observed that the coefficients of $\rho(\omega_2)$ belongs to $\C^{k-2,\alpha}_{-\tau-2}$ and $(\tr_{\omega_2}-\tr_{\omega_1}) \in \C^{k,\alpha}_{-\tau}$, which implies that
\begin{align*}
(\tr_{\omega_2}-\tr_{\omega_1})\rho(\omega_2) \in \C^{k-2,\alpha}_{-2\tau-2}.
\end{align*}
Hence, we can deduce that $\varphi$ satisfies the equation, $\Delta_{\omega_1}^2 \varphi = f $ with $f\in \C^{k-2,\alpha}_{-2-2\tau}$ on $X$ with $-2-2\tau < -2n$. By solving the Laplacian on asymptotic conical manifolds, proposition \ref{laequac} (iii) imples
\begin{align} \label{expanpoten}
\Delta_{\omega_1} \varphi = C r^{2-2n} + \phi, \qquad \phi \in \C^{k,\alpha}_{-2\tilde{\tau}} \text{ with } \tilde{\tau} = \min\{\tau, - d^-_1/2\},
\end{align}
where $d_1^-$ is the second negative exceptional weight in $D$ defined in $\ref{excd}$ with $d_1^- < d_0^- = 2-2n$.
By proposition  \ref{massf1}, combining with (\ref{acconnform1}), we have
\begin{align}
\nonumber \m (\omega_2)- \m (\omega_1) = \lim_{r\rightarrow \infty} \frac{1}{(2n-1) (n-1)! \vol(L)}&\bigg[ \int_{L(r)} \theta_{\omega_2} \wedge \omega_2^{n-1}-\int_{L(r)} \theta_{\omega_1} \wedge \omega_1^{n-1}\bigg] \\ \nonumber
= \lim_{r\rightarrow \infty} C(n , L)  &\bigg[ \int_{L(r)}\theta_{\omega_2} \wedge (\omega_2^{n-1}-\omega_1^{n-1}) \\ \nonumber
& \qquad - \int_{L(r)}  (\theta_{\omega_2} - \theta_{\omega_1}) \wedge \omega_1^{n-1} \bigg]. \\ \nonumber
=\lim_{r\rightarrow \infty}C(n, L) & \bigg[ \int_{L(r)} \theta_{\omega_2} \wedge (\omega_2^{n-1}-\omega_1^{n-1}) \\ 
& \qquad - \frac{1}{2}\int_{L(r)} d^C \log \frac{(\omega_1+dd^C \varphi)^{n}}{\omega_1^n} \wedge \omega_1^{n-1} \bigg], \label{masscon}
\end{align}
where $C(n,L) $ is a constant given in proposition \ref{massf1}. Since $\varphi \in \C^{k,\alpha}_{-\tau}$ and by (\ref{acconnform1}), the coefficients of $\theta_{\omega_2}$ with respect to a local asymptotic frame are of class $\C^{k-1,\alpha}_{-\tau-1}$, then we have
\begin{align} \label{masscon1}
\int_{L(r)} \theta_{\omega_2} \wedge (\omega_2^{n-1}-\omega_1^{n-1}) = \int_{L(r)} \theta_{\omega_2}\wedge dd^C \varphi \wedge (\omega_1^{n-2} + \omega_1^{n-3} \wedge \omega_2 + \ldots + \omega^{n-2}_2 ) = O(r^{-2\epsilon}),
\end{align}
and 
\begin{align} \label{masscon2}
\int_{L(r)} d^C \log \frac{(\omega_1+dd^C \varphi)^n}{\omega_1^n} \wedge \omega_1^{n-1}= \int_{L(r)} d^C (\Delta_{\omega_1} \varphi) \wedge \omega_1^{n-1} +O(r^{-2\epsilon}).
\end{align}
By inserting (\ref{masscon1}) and (\ref{masscon2}), the difference of mass is
\begin{align*}
\m (\omega_2) - \m (\omega_1) =  -\lim_{r\rightarrow \infty} C(n,L)  \int_{L(r)} d^C(\Delta_{\omega_1} \varphi) \wedge \omega_1^{n-1}
\end{align*}
Notice that the standard volume form on the link $L$ of the model cone $C_L$ is given by $d \vol_{g_L}  = \eta \wedge \omega_0^{n-1}|_{L(1)}$, where  $\eta =r^{-1} J_0  dr$ is the contact 1-form of the Sasakian manifold $L$. Then,
\begin{equation} \label{masscoefac}
\begin{split}
 -\int_{L(r)} d^C(\Delta_{\omega_1} \varphi) \wedge \omega_1^{n-1} &=  -C  \int_{L(r)} J d(r^{2-2n}) \wedge \omega_1^{n-1} \\
&= C\int_{L(r)} (2n-2) r^{1-2n} (r \eta) \wedge \omega_0^{n-1} + O(r^{-\tau})  \\
&= C (2n-2){\vol (L)} + O(r^{-\tau}),
\end{split}
\end{equation}
where $C$ is the constant in (\ref{expanpoten}). According to the mass formula in theorem \ref{massf2}, the mass only depends on the K\"ahler class, the first Chern class and scalar curvature; hence, $\m(\omega_1) =\m(\omega_2)$. Then, we have $C=0$, which completes the proof of theorem \ref{expanthm} (i).

\subsection{ Proof of Theorem \ref{expanthm} (ii)}
For the proof theorem \ref{expanthm} (ii), in the case of $\dim_{\CC} X_{\infty} = 2$, noticing that the link of Ricci-flat K\"ahler cone is an Sasakian Einstein manifold of real dimensional 3, $L$ has constant sectional curvature. Thus, $C_L$ is biholomorphic to $\CC^2/\Gamma$ for some finite subgroup $\Gamma \subseteq U(2)$. Hence, the discussion can be reduced to ALE cases. Since $H^2(X_\infty) = H^2(S^3/\Gamma) = 0$, $\omega-\omega_{euc}$ is exact real $(1,1)$ form on $X_\infty$ with decay rate $-\tau$. According to Theorem \ref{ddclem} (ii), there exists $\varphi \in \C^{k+2, \alpha}_{2-\tau}$ such that
\begin{align} \label{ddcsfkform}
\omega - \omega_{euc} =  dd^c \varphi + O(r^{-4}).
\end{align}
If $\omega$ is a scalar-flat K\"ahler metric, then by formula of scalar curvature, in the end $X_\infty$, 
\begin{align} \label{scalflatcon}
0= \tr_\omega \rho &= \tr_\omega d A_J - \frac{1}{2} \tr_\omega dd^c \log \frac{\omega^n}{\omega_0^n} + O(r^{-2\tau-2}) \nonumber\\
& = \tr_\omega dA_J -\frac{1}{2} \Delta_{\omega}^2 \varphi + O(r^{\max\{-2\tau-2, -6\}}).
\end{align}
Recall that, in complex dimension $2$, the complex structure of ALE manifolds decays to the standard complex structure in Euclidean space as $J-J_0 = O(r^{-3})$ (see  \cite[Proposition 4.5]{hein2016mass}). Since $A_J = 2 \im \alpha$, where $\alpha$ is given in Lemma \ref{01conn}, then its differential has $dA_J = O(r^{-5})$. Then, the equation (\ref{scalflatcon}) can be rewritten as,
\begin{align}
    \Delta^2_\omega \varphi = f, \quad \text{ for } f \in \C^{k-2,\alpha}_{-2\tilde{\tau}-2},
\end{align}
where $\tilde{\tau} = \min\{\tau, 3/2\}$. Also, by Proposition \ref{laequac} (iii), $\Delta_\omega \varphi = C r^{2-2n} + \phi $ with $\phi \in \C^{k,\alpha}_{-2\tilde{\tau}}$. According to Proposition \ref{massf1}, we have
\begin{align*} 
\m (\omega) &= \lim_{r \rightarrow \infty} \frac{1}{3\vol(L)} \int_{L(r)} \theta_\omega \wedge \omega^{n-1}\\
&= \lim_{r \rightarrow \infty}  \frac{1}{3\vol(L)}  \int_{L(r)} \Big( A_J - \frac{1}{2} d^C  \log \frac{(\omega_0 + dd^C \varphi)^n}{\omega_0^n  }  \Big)  \wedge \omega^{n-1} \\
&= -\lim_{r \rightarrow  \infty} \frac{1}{3\vol(L)}   \int_{L(r)}  d^C (\Delta_{\omega} \varphi) \wedge \omega^{n-1}
\end{align*}
Then, repeating the calculation in (\ref{masscoefac}), we have the constant $C$,
\begin{align*}
    C = 3 \m(\omega).
\end{align*}
Notice that $\displaystyle \Delta_0 \log r = \frac{2}{r^2}$, we have the solution
\begin{align*}
    \varphi = \frac{3 \m(\omega)}{2} \log r + O(r^{-2\tilde{\tau}+2}).
\end{align*}

In the case of $\dim_{\CC} X_{\infty} \geq 3$, the asymptotic condition of the K\"ahler form on $X_\infty$, $\omega - \omega_0 = O(r^{-\tau})$ with $\tau = n-1 +\epsilon > 2$,  implies that $\omega- \omega_0$ is an exact form. One can check it by identifying $L\times [T_0, \infty)$ with $X_\infty$. Define a new parameter $t = \log r$ on cylinder $L\times [T_0, \infty) $. Consider the cylinder metric $g_{cyl} = dt^2 + g_L$, then the relation with the conical metric is  $g_{cyl}= e^{-2t} g_0 $. Therefore, $|\omega - \omega_0|_{g_{cyl}} = O(e^{2-\tau})$. Now, let $\omega-\omega_0 = \eta_2 + \eta_1 \wedge dt$. If we integrate the form in $t$ direction, $\eta  = -\int_\infty^t \eta_1  \wedge dt$, then
\begin{align*}
    d\eta &= d_L \eta - \partial_t \eta \wedge dt \\
    & = - \int_\infty^t d_L \eta_1 \wedge dt + \eta_1 \wedge dt \\
    & = \int^t_\infty \partial_t \eta_2 \wedge dt + \eta_1 \wedge dt = \eta_2 + \eta_1 \wedge dt,
\end{align*}
where the third equality is from the closedness of $\omega-\omega_0$. According to Theorem \ref{ddclem} (ii), there exists a real function $\varphi \in \C^{k+2,\alpha}_{2-\tau}$ such that $\omega - \omega_0 = i\partial \overline{\partial} \varphi + O(r^{-2n})$. Based on the formula (\ref{scalarcurv}) and $R(\omega)=0$, we have the equation $\Delta^2_\omega \varphi = f$, for $f \in \C^{k-2,\alpha}_{-2\tau'-2}$, where the decay rate of $f$ is determined by the decay rate of $dA_J$ (by assumption (\ref{cxdec})), the dimension $n$ and $\tau$; in particular,  $\tau' = \min\{n, \tau, n-1+ \epsilon'\}$. Also, by proposition \ref{laequac} (ii), $\Delta_\omega \varphi = C r^{2-2n} + \phi \ (**)$ with $\phi \in \C^{k,\alpha}_{-2\tilde{\tau}}$, $\tilde{\tau} = \min\{\tau', -d^-_1/2\}$. Combining proposition \ref{massf1}, (\ref{acconnform1}) and (\ref{masscoefac}), we can obtain a formula for the constant $C$ in $(**)$, 
\begin{align*}
    C = \frac{2n-1}{n-1} m(\omega).
\end{align*}
Notice that $\Delta_0 r^{4-2n} = 2(4-2n) r^{2-2n}.$
Therefore, we have the expansion of the solution of $(**)$, 
\begin{align*}
    \varphi = \frac{C}{ 2(4-2n)} r^{4-2n} + O(r^{2-2\tilde{\tau}})
\end{align*}
Noting the the standard K\"ahler form on $C_L$ is $\omega_0 = dd^c r^2 /2$, the K\"ahler form $\omega$ admits the following expansion, 
\begin{align*}
     \omega = \frac{1}{2} dd^c r^2 + \frac{(2n-1) \m(\omega)}{2(4-2n)(n-1)} dd^c r^{4-2n} + O(r^{-2\tilde{\tau}}),
\end{align*}
which completes the proof.

\section{Positive Mass Theorem on Resolution Spaces of Calabi-Yau Cones} \label{pmts}

We conclude the paper by proving the positive mass theorem based on the generalized ADM mass defined in (\ref{massdef}). 

Let $(C_L, J_0, g_0)$ be a Ricci-flat K\"ahler cone with only one singularity at the vertex $O$. Recall the resolution of the Ricci-flat K\"ahler cone is a smooth complex manifold $X$ with a proper map $\pi : X \rightarrow C_L$ such that the map $\pi: X\backslash E \rightarrow C_L \backslash \{O\} $ is  biholomorphic, where $E= \pi^{-1}(O)$ is the exceptional divisor. The positive mass theorem does not always hold for resolution spaces of Ricci-flat K\"ahler cone cones, for instance, the counter example constructed by LeBrun \cite{cmp/1104162166}. So it is reasonable to consider a special class of resolution spaces. In particular, we say the isolated singularity $O \in C_L$ is \textit{canonical} if the resolution spaces $(X, \pi)$ satisfies,
\begin{align} \label{kx}
    K_X = \pi^* K_{C_L} + \sum_{i} a_i E_i,
\end{align}
where $E_i$ are the irreducible hypersurfaces and $a_i \geq 0$. Then, we can prove the following positive mass theorem on canonical resolution spaces of Calabi-Yau cones. 

\begin{thm} Let $(C_L, J_0, g_0)$ be a Ricci-flat K\"ahler cone such that the only singularity $O \in C_L$ is canonical. Suppose that $X$ admits an asymptotically conical K\"ahler metrics $g$ with decay rate $-\tau$, $\tau = n-1+\epsilon$ $(\epsilon >0)$. If $(X,g)$ has scalar curvature $R \geq 0$, then the mass $\m(X,g) \geq 0$, and equals zero only if $(X,J,g)$ is a crepant resolution of $C_L$ with a scalar-flat K\"ahler metric $g$.  
\end{thm}
  
\begin{proof} According to the assumption that $s \geq 0$, the integral of scalar curvature in the mass formula  (\ref{massf3}) is nonnegative, and equals zero only if $g$ is scalar-flat. It suffices to prove the first term of the mass formula is also nonnegative.

Firstly, recalling the map $\iota: H_c^2(X) \rightarrow H^2 (X)$ induced by inclusion of cohomology classes, we prove that $[\omega] \in H^2(X)$ has a preimage in $H_c^2(X)$. For $\dim_\CC X =2$, the Calabi-Yau cone must be $\CC^2/ \Gamma$ with $\Gamma$ is a finite subgroup of $U(2)$. The fact $H^{1}(S^3/ \Gamma) = H^2(S^3/ \Gamma)=0$ implies that $\iota$ is an isomorphism. For $\dim_\CC X \geq 3$, since we have the decay condition $\omega - \frac{1}{2}dd^C r^2 = O(r^{-\tau})$ with $\tau>2$ in $X_\infty$, the lemma \ref{cptclasslm} implies that $\omega - \frac{1}{2} dd^C r^2 $ is an exact form in $X_\infty$, written by $d \theta$. Let $f$ be a smooth function identically equal to $1$ near infinity and vanishing in a compact set of $X$, for instance, the function defined in (\ref{cutoff}), then $\omega - d (f (d^C r^2/4 + \theta))$ has a compact support.

Notice that $c_1 (X) = - c_1(K_X)$, then,  
\begin{align*}
    -\langle \iota^{-1} c_1, [\omega]^{n-1} \rangle = - \langle \iota^{-1} c_1, (\iota^{-1}[\omega])^{n-1} \rangle = \langle  c_1(K_X),(\iota^{-1} [\omega])^{n-1} \rangle.
\end{align*}
According to (\ref{kx}) and Poincar\'{e} duality, we have,
\begin{align*}
    \langle c_1(K_X), (\iota^{-1}[\omega])^{n-1} \rangle = \sum_{i}a_i\int_{E_i} \omega^{n-1}\geq 0.
\end{align*}
The above inequality only if $a_i =0$ for all $i$, which completes the proof. 

\end{proof}

\begin{ex}
Let $D$ be a compact Fano manifold with $\dim_\CC D =n-1$ and $K_D^\times$, the space by shrinking the zero section of the canonical line line bundle of $D$. In this example, we consider the standard resolution of $(K_D^{\alpha})^\times$, where $\alpha$ is a positive fractional number. Firstly, we check that $K_D^\times$ is a Calabi-Yau cone. Let $U$ be a local chart in $D$ as well as a local trivialization of $K_D$. We assume that $U\times \CC \subseteq K_D$ has a holomorphic coordinate system $\{ z_1, \ldots, z_{n-1}, u \}$. There exists a local holomorphic $(n,0)$ form on $U \times \CC$, $ \Omega_U = dz_1\wedge \ldots \wedge dz_{n-1} \wedge d u$ and $\Omega_U$ can be extended naturally to be a global nonvanishing holomorphic $(n,0)$ form on $K_D$. To see this, let $V$ be another local chart of $D$  with $K_D|_V = V\times \CC$ and $\{w_1,\ldots, w_{n-1} , v\}$, its holomorphic coordinates. Recall that the transition function of $K_D$ is given by $ v= u \cdot J(\partial z/ \partial w)$, then, 
\begin{align*}
dw_1\wedge \ldots \wedge dw_{n-1} \wedge d v = J(\partial z/ \partial w) dw_1 \wedge \ldots \wedge dw_{n-1}  \wedge du = dz_1\wedge \ldots \wedge dz_1 \wedge d u.
\end{align*}
Hence, we have a global nonvanishing holomorphic  $(n,0)$ form, $\Omega$ on $K_D$. By restricting $\Omega$ to the space away from the zero section of $K_D$, we show that $K_D^\times $ is a Calabi-Yau cone. Similarly, we can find a global holomorphic $(n,0)$ form, $\Omega$, on $K_D^{\alpha}$. Let $U$ be a local chart in $D$ with $(K_D)^\alpha |_{U} = U \times \CC$ and $\{z_1, \ldots, z_{n-1}, t \}$ be a holomorphic coordinates of $\{U \times \CC\}$. Then, locally, $\Omega$ can be written as
\begin{align*}
    \Omega|_{U} = \frac{1}{\alpha} t^{\frac{1-\alpha}{\alpha}} dz_1 \wedge \ldots \wedge dz_{n-1} \wedge dt.
\end{align*}
Also, by restricting $\Omega$ away from the zero section of $K_D^\alpha$, we obtain a nonvanishing holomorphic (possibly multi-valued) $(n,0)$ form on $(K_D^\alpha)^\times$. Note that $\Omega$ is multi-valued if the order $(1-\alpha)/\alpha$ is not an integer. Therefore, for the resolution map by shrinking zero section, $\pi: K^\alpha_D \rightarrow (K^\alpha_D)^\times$, we have the following adjunction type formula,
\begin{align*}
   K_{ K_D^{\alpha}} = \pi^* K_{ (K_D^\alpha)^{\times}} + \frac{1-\alpha}{\alpha} D
\end{align*}
Hence,  $K^\alpha_D$ satisfies the positive mass theorem if and only if $\alpha \leq 1$. 
\end{ex}

\appendix
\section{Linearization of Scalar Curvature and Mass on AE manifolds}

The goal of this appendix is to complete the detailed calculation of linearization of scalar curvature. Based on the calculation, we will see that, on AE manifolds, the ADM mass can be interpreted as the integral of linearization of scalar curvature over the sphere at infinity.

Let $X$ be a Riemannian manifold with a base metric $\overline{g}$. If we assume that $g$ is another metric over $X$, then the linearization of scalar curvature $DR(g)$ with respect to $(X, \overline{g})$ is defined to be 
\begin{align*}
    DR(g) = \frac{d}{dt} \Big|_{t=0} R(t) 
\end{align*}
where $R(t)$ is the scalar curvature of $g(t) = \overline{g} + t(g- \overline{g})$. Let $\{\xi_i, 1\leq i \leq n\}$ be a system of local frame on $X$. Let $\nabla$ (resp. $\overline{\nabla}$) be the Levi-Civita connection of $g$ (resp. $\overline{g}$) and $\Gamma$ (resp. $\overline{\Gamma}$) be the Christoffel symbol of $\nabla$ (resp. $\overline{\nabla}$). Also the difference of two connections is denoted by $A = \nabla - \overline{\nabla}$. Then, the Ricci curvature of $g$ and $\overline{g}$ has the following relation,
\begin{lem}
Let $R_{ij}$ and $\overline{R}_{ij}$ be the Ricci curvature of $g$ and $\overline{g}$ with respect to the local frame system $\{\xi_i, 1\leq i\leq n\}$ and its dual $\{\omega^i, 1\leq i \leq n\}$, then we have
\begin{align*}
R_{ij} = \overline{R}_{ij} + (\overline{\nabla}_k A^k_{ij} - \overline{\nabla}_i A^k_{kj}) + (A^k_{k l} A^{l}_{ij} - A^k_{jl} A^l_{ik}) + A^k_{lj} \omega^l([\xi_k, \xi_i]).
\end{align*}
\end{lem}
\begin{proof}
Notice that $R_{ij} = \omega^k (R(\xi_k, \xi_i)\xi_j) = \omega^k (\nabla_{\xi_{k}} \nabla_{\xi_i} \xi_j - \nabla_{\xi_i} \nabla_{\xi_k} \xi_j)$ and $\overline{R}_{ij}$ admits a similar formula. Then, we have
\begin{align}\label{linofsp1}
\begin{split}
    R_{ij}- \overline{R}_{ij} = &\ \omega^k (\nabla_{\xi_{k}} \nabla_{\xi_i} \xi_j - \nabla_{\xi_i} \nabla_{\xi_k} \xi_j) - \omega^k (\overline{\nabla}_{\xi_{k}} \overline{\nabla}_{\xi_i} \xi_j - \overline{ \nabla}_{\xi_i} \overline{\nabla}_{\xi_k} \xi_j) \\
     =&\ \omega^k \Big[ \overline{\nabla}_{\xi_k} (\nabla_{\xi_i} - \overline{\nabla}_{\xi_i}) \xi_j + (\nabla_{\xi_k}-\overline{\nabla}_{\xi_k}) \overline{\nabla}_{\xi_i} \xi_j + (\nabla_{\xi_k} - \overline{\nabla}_{\xi_k})(\nabla_{\xi_i} - \nabla_{\xi_i}) \xi_j \Big] \\
    & -\omega^k \Big[ \overline{\nabla}_{\xi_i} (\nabla_{\xi_k} - \overline{\nabla}_{\xi_k}) \xi_j + (\nabla_{\xi_i}-\overline{\nabla}_{\xi_i}) \overline{\nabla}_{\xi_k} \xi_j + (\nabla_{\xi_i} - \overline{\nabla}_{\xi_i})(\nabla_{\xi_k} - \nabla_{\xi_k}) \xi_j \Big]\\
    =& \ \big(\xi_k (A_{ij}^k)  + A^l_{ij} \overline{\Gamma}^k_{lk} + A^k_{lk} \overline{\Gamma}^l_{ij} + A^l_{ij} A^k_{lk} \big) - \big( \xi_i (A_{kj}^k) + A^l_{kj} \overline{\Gamma}^k_{il} + A^k_{li} \overline{\Gamma}^l_{kj} + A^l_{kj} A^k_{li}  \big) \\
    =& \ \xi_k (A_{ij}^k) - \xi_i (A_{kj}^k) + A^l_{ij} \overline{\Gamma}^k_{lk} + A^k_{lk} \overline{\Gamma}^l_{ij} - A^l_{kj} \overline{\Gamma}^k_{il} - A^k_{li} \overline{\Gamma}^l_{kj} \\ &\ +A^l_{ij} A^k_{lk} - A^l_{kj} A^k_{li}.
    \end{split}
\end{align}
Also we can compute $\overline{\nabla} A$, then
\begin{align}\label{linofsp2}
\begin{split}
   \overline{\nabla}_k A^k_{ij} -  \overline{\nabla}_i A^k_{kj} = & \ \xi_k (A_{ij}^k) +A^l_{ij} \overline{\Gamma}_{lk}^k - A^k_{lj} \overline{\Gamma}_{ki}^l - A^k_{il} \overline{\Gamma}^l_{kj} \\ 
   & - \xi_i (A_{kj}^k) - A^l_{kj} \overline{\Gamma}^k_{il} + A^k_{lj} \overline{\Gamma}^l_{ik} + A^k_{lk} \overline{\Gamma}^l_{ij}\\
    = & \ \xi_k (A_{ij}^k) - \xi_i (A_{kj}^k) + A^l_{ij} \overline{\Gamma}^k_{lk} + A^k_{lk} \overline{\Gamma}^l_{ij} - A^l_{kj} \overline{\Gamma}^k_{il} - A^k_{li}\overline{\Gamma}^l_{kj}\\
    & + A^k_{lj} \overline{\Gamma}^l_{ik} - A^k_{lj} \overline{\Gamma}^l_{ki}.
\end{split}
\end{align}   
Notice that $A$ is a symmetric tensor, but $\overline{\Gamma}$ is not and 
\begin{align*}
    \overline{\Gamma}^{l}_{ik} - \overline{\Gamma}^l_{ki} = \omega^l ([\xi_i, \xi_k]). 
\end{align*}
By plugging (\ref{linofsp2}) into (\ref{linofsp1}), we have
\begin{align*}
    R_{ij}- \overline{R}_{ij} =  \overline{\nabla}_k A^k_{ij} -  \overline{\nabla}_i A^k_{kj} -  A^k_{lj}\omega^l ([\xi_i, \xi_k]) +A^l_{ij} A^k_{lk} - A^l_{kj} A^k_{li},
\end{align*}
which completes the proof.
\end{proof}

\begin{lem} \label{linofs1} Let $DR(g)$ be the linearization of scalar curvature with respect to $(X, \overline{g})$. 
Then, on the local frame system $\{\xi_i, 1\leq i\leq n\}$,
\begin{align*}
    DR(g) = \langle  \Ric_{\overline{g}}, g-\overline{g} \rangle_{\overline{g}} + \overline{\nabla}^j \overline{ \nabla}^i g_{ij}  - \overline{\Delta} \tr_{\overline{g}} g.
\end{align*}
\end{lem}
\begin{proof}
Firstly, we derive a formula for $A$. Notice that
\begin{align*}
    \overline{\nabla}_i g_{jk} & = \xi_i (g_{jk}) - g( \overline{\nabla}_{\xi_i} \xi_j , \xi_k ) - g( \xi_j , \overline{\nabla}_{\xi_i} \xi_k ) \\
    & = g\big( (\nabla_{\xi_i} - \overline{\nabla}_{\xi_i}) \xi_j , \xi_k \big) +g \big( \xi_j ,( \nabla_{\xi_i} -\overline{\nabla}_{\xi_i}) \xi_k \big) \\
    & = g_{lk} A^l_{ij} + g_{jl} A^l_{ik}.
\end{align*}
By permutation of the indice $i,j,k$, we have
\begin{align}\label{fofa}
    A^k_{ij} = \frac{1}{2} g^{kl} (\overline{\nabla}_i g_{jl} + \overline{\nabla}_j g_{il} - \overline{\nabla}_l g_{ij})
\end{align}
Let $A(t) = \nabla_t - \overline{\nabla}$ and $\nabla_t$ be the Levi-Civita connection of $g(t) = \overline{g} + t(g - \overline{g})$, then
\begin{align*} 
    R(t) = g^{ij}(t) \big( \overline{R}_{ij} + \overline{\nabla}_k A^k_{ij}(t) -  \overline{\nabla}_i A^k_{kj} (t) -  A^k_{lj}(t) \omega^l ([\xi_i, \xi_k]) +A^l_{ij}(t) A^k_{lk}(t) - A^l_{kj}(t) A^k_{li}(t)  \big)
\end{align*}
Notice that $A(0) =0$, we have
\begin{align}
\begin{split}\label{linofsp3}
    \frac{d}{dt} \Big|_{t=0} R(t) = & \overline{R}_{ij} \frac{d}{dt} \Big|_{t=0} g^{ij}(t) \\ & + \overline{g}^{ij}  \frac{d}{dt} \big|_{t=0} \big(\overline{\nabla}_k A^k_{ij}(t) -  \overline{\nabla}_i A^k_{kj} (t)\big) - \omega^l ([\xi_i, \xi_k]) \overline{g}^{ij} \frac{d}{dt} \Big|_{t=0} A^k_{lj}(t).
\end{split}
\end{align}
Then, we compute each term in (\ref{linofsp3}) separately.  
\begin{align*}
    \frac{d}{dt} \Big|_{t=0} A^k_{lj}(t) & = \frac{1}{2} \overline{g}^{km} (\overline{\nabla}_l  \dot{g}_{jm} + \overline{\nabla}_j \dot{g}_{ml} - \overline{\nabla}_m \dot{g}_{jl})(0)\\
    & = \frac{1}{2} \overline{g}^{km} (\overline{\nabla}_l {g}_{jm} + \overline{\nabla}_j {g}_{ml} - \overline{\nabla}_m {g}_{jl})
\end{align*}
Inserting into the third term of (\ref{linofsp3}),
\begin{align*}
    \omega^l ([\xi_i, \xi_k]) \overline{g}^{ij} \frac{d}{dt} \Big|_{t=0} A^k_{lj}(t) & = \frac{1}{2}\overline{g}^{ij}  \overline{g}^{km} (\overline{\nabla}_{[i,k]} {g}_{jm} + \overline{\nabla}_j {g}_{m[i,k]} - \overline{\nabla}_m {g}_{j[i,k]}) =0
\end{align*}
The above identity is vanishing because of the symmetry of indices; in particular, by permuting $(ij)$ and $(km)$, we can check that each term in above identity is vanishing. Notice that
\begin{align*}
  \overline{g}^{ij} \frac{d}{dt} \Big|_{t=0} \overline{\nabla}_{i}  A^k_{kj}(t) 
  & =  \frac{1}{2} \overline{g}^{ij} \overline{\nabla}_i \big( \overline{g}^{kl}(\overline{\nabla}_j g_{kl} + \overline{\nabla}_k g_{lj} - \overline{\nabla}_l g_{jk}) \big)\\
  & = \frac{1}{2} \overline{g}^{ij} \overline{\nabla}_{i}  \overline{\nabla}_j \big(\overline{g}^{kl} g_{kl} \big) \\ 
  & = \frac{1}{2} \Delta \tr_{\overline{g}} g
\end{align*}
and 
\begin{align*}
   \overline{g}^{ij} \frac{d}{dt} \Big|_{t=0} \overline{\nabla}_{k}  A^k_{ij}(t) 
  & = \frac{1}{2}\overline{g}^{ij}  \overline{\nabla}_k \big( \overline{g}^{kl}(\overline{\nabla}_j g_{il} + \overline{\nabla}_i g_{lj} - \overline{\nabla}_l g_{ji}) \big)\\
  & = \overline{g}^{ij}  \overline{\nabla}_k \big( \overline{g}^{kl}\overline{\nabla}_j g_{il} \big) - \frac{1}{2} \overline{g}^{ij}  \overline{\nabla}_k \big( \overline{g}^{kl} \overline{\nabla}_l g_{ij} \big) \\
  & =  \overline{\nabla}^j \overline{ \nabla}^i g_{ij} - \frac{1}{2} \overline{\Delta} \tr_{\overline{g}} g,
\end{align*}
combining with $ d/dt|_{t=0} g^{ij} (t) = \overline{g}^{ik}\overline{g}^{jl} (g-\overline{g})_{kl} $ and \ref{linofsp3}, then we have,
\begin{align*}
    \frac{d}{dt} \Big|_{t=0} R(t) = \langle  \Ric_{\overline{g}}, g-\overline{g} \rangle_{\overline{g}} + \overline{\nabla}^j \overline{ \nabla}^i g_{ij}  - \overline{\Delta} \tr_{\overline{g}} g,
\end{align*}
which completes the proof.
\end{proof}
Based on lemma \ref{linofs1}, we can quickly derive the ADM mass on AE manifolds from the integral of linearization of scalar curvature. Let $g$ be a Riemannian metric in $\RR^n$ and $DR(g)$ be the linearization of scalar curvature with respect to the Euclidean metric on $\RR^n$. Then, the integral of $DR(g)$ is given by
\begin{align}
    \int_{\RR^n} DR(g) d\vol_{euc} 
    & = \lim_{r \rightarrow \infty} \int_{S^{n-1}(r)} \big( \overline{\nabla}^j g_{ij} - (\tr_{g_0} g) \big) n^i d \vol_{S^{n-1} (r)} \nonumber \\
    & = \lim_{r \rightarrow \infty} \int_{S^{n-1}(r)} (g_{ij,j}- g_{jj, i}) n^i d\vol_{S^{n-1} (r)}. \label{massae}
\end{align}
Notice that the formula (\ref{massae}) only depends on the information of metrics at infinity. Hence, up to a constant factor, (\ref{massae}) can be defined to be the mass on AE manifolds.

\bibliographystyle{plain}
\bibliography{Reference}

\end{document}